\documentclass[11pt,letterpaper]{amsart}
\usepackage{amsfonts, amsthm, amssymb, amsmath, stmaryrd}
\usepackage{mathrsfs,array}
\usepackage{eucal,color,times,enumerate,accents}
\usepackage{tikz-cd}
\usepackage{bbm}
\usepackage[utf8]{inputenc}
\usepackage{pgfplots}
\pgfplotsset{compat=newest}
\usepackage{lipsum}
\usepackage{graphicx,txfonts}
\usepackage{rotating}
\usepackage{enumitem}
\usepackage{enumerate}
\usepackage{comment}

\usepackage[backend=biber,style=alphabetic,sorting=nyt,doi=false,isbn=false,url=false]{biblatex}
\usepackage{footmisc}
\usepackage{textcomp}
\usepackage[bookmarksnumbered,bookmarksopen]{hyperref}
\hypersetup{colorlinks, linkcolor=black, citecolor=black}
\usepackage{cleveref}
\usepackage[colorinlistoftodos]{todonotes}
\usepackage[left=4cm,right=4cm,top=3cm,bottom=3cm]{geometry}
\usepackage{mathtools}

\newtheorem{theorem}[equation]{Theorem}

\newtheorem{proposition}[equation]{Proposition}
\newtheorem{lemma}[equation]{Lemma}

\newtheorem{corollary}[equation]{Corollary}

\numberwithin{equation}{subsection} 
\newcommand{\heartful}{\ensuremath\varheartsuit}

\theoremstyle{definition}
\newtheorem{definition}[equation]{Definition}

\newtheorem{remark}[equation]{Remark}
\newtheorem{example}[equation]{Example}

\usepackage{tikz}
\usepackage{amsmath}
\usepackage{tikz-cd}
\usepackage{rotating}
\usetikzlibrary{arrows,shapes,positioning}
\usetikzlibrary{decorations.markings}
\tikzstyle arrowstyle=[scale=1]
\usepackage{lipsum}
\newcommand{\B}{\mathcal{B}}
\newcommand{\C}{\mathcal{C}}
\usepackage{url}
\usepackage{imakeidx}
\makeindex[title=Index of definitions and notations,intoc]
\newcommand{\Z}{\mathbb{Z}}
\newcommand{\Q}{\mathbb{Q}}
\renewcommand{\H}{\mathrm{H}}

\newcommand{\Qp}{\mathbb{Q}_p}

\newcommand{\rank}{\mathrm{rank}}
\newcommand{\Dim}{\mathrm{dim}}
\newcommand{\Ker}{\mathrm{Ker}}
\newcommand{\Image}{\mathrm{Im}}
\newcommand{\Coker}{\mathrm{Coker}}

\newcommand{\Gal}{\mathrm{Gal}}
\newcommand{\Hom}{\mathrm{Hom}}
\newcommand{\Id}{\mathrm{Id}}

\newcommand{\Gr}{\mathrm{Gr}}
\newcommand{\et}{\textrm{\'et}}
\newcommand{\Ainf}{A_\mathrm{inf}}

\newcommand{\dr}{\mathrm{dR}}
\newcommand{\Os}{\mathcal{O}}
\newcommand{\Ocpflat}{\mathcal O_{\mathbb C_p^{\flat}}}

\newcommand{\Ocp}{\mathcal O_{\mathbb C_p}}
\newcommand{\A}{\mathcal{A}}
\newcommand{\End}{\mathrm{End}}

\newcommand{\dR}{\mathrm{dR}}

\newcommand{\Sh}{\mathrm{Sh}}
\newcommand{\sh}{\mathrm{sh}}

\newcommand{\F}{\mathbb{F}}
\newcommand{\gb}{\mathrm gb}

\DeclareMathOperator{\Br}{Br}
\DeclareMathOperator{\fil}{fil}
\DeclareMathOperator{\rsw}{rsw}
\DeclareMathOperator{\ev}{ev}
\DeclareMathOperator{\Det}{det}
\DeclareMathOperator{\TS}{TS}
\DeclareMathOperator{\NS}{NS}

\DeclareMathOperator{\Tr}{Tr}
\makeatletter
\tikzcdset{
  eq node/.style={
    commutative diagrams/math mode=false, anchor=center},
  eq/.style={
    phantom,
    /tikz/every to/.append style={
      edge node={node[commutative diagrams/eq node]
        {\@eqnswtrue\make@display@tag\ltx@label{#1}}}}}}
\makeatother
\usetikzlibrary{cd}

\begin{document}
\title{Wild Brauer classes via prismatic cohomology}
\date{\today}
\author{Emiliano Ambrosi}
\address{Institut de Recherche Math\'{e}matique Avanc\'{e}e (IRMA)\\
 Universit\'{e} de Strasbourg\\
 Strasbourg\\
 France}
\email{eambrosi@unistra.fr}

\author{Rachel Newton}
\address{Department of Mathematics\\
King's College London\\
Strand\\ London\\ 
UK}
\email{rachel.newton@kcl.ac.uk}

\author{Margherita Pagano}
\address{Department of Mathematics\\ South Kensington Campus\\
Imperial College London\\
London\\ 
UK}
\email{m.pagano@imperial.ac.uk}

\begin{abstract}
Let $K$ be a finite extension of $\mathbb{Q}_p$ and $X$ a smooth proper $K$-variety with good reduction. Under a mild assumption on the behaviour of Hodge numbers under reduction modulo $p$, we prove that the existence of a non-zero global 2-form on $X$ implies,  after a finite extension of $K$, the existence of $p$-torsion Brauer classes with surjective evaluation map. This implies that any smooth proper variety over a number field which satisfies weak approximation over all finite extensions has no non-zero global 2-form. The proof is based on a prismatic interpretation of  Brauer classes with eventually constant evaluation, and a Newton-above-Hodge result for the mod $p$ reduction of prismatic cohomology. This generalises work of Bright and the second-named author beyond the ordinary reduction case.  
\end{abstract}
\maketitle
\tableofcontents
\section{Introduction}
Let $K$ be a finite extension of $\Q_p$, with residue field $k$. Let $X$ be a smooth proper variety over $K$ with good reduction. In this paper we use recent developments in $p$-adic cohomology to demonstrate the existence (and in some cases construction) of $p$-torsion Brauer classes of arithmetic interest. 
\subsection{Brauer classes with non-constant evaluation map}
\subsubsection{Main result}
Recall that each $\A\in \Br(X)$ induces, for every finite field extension $L/K$, an evaluation map 
$$\ev^L_\A :X(L)\rightarrow \Br(L)\simeq \Q/\Z$$
sending $Q\in X(L)$ to $Q^*\A$. Consider the following very mild condition on the Hodge numbers of the special and generic fibres: 
  \begin{equation}\tag*{\heartful}\label{eq : equalityhodgenumber}
        \Dim_{k}(\H^i(Y,\Omega_{Y/k}^j))=	\Dim_{K}(\H^i(X,\Omega_{X/K}^j))\quad \text{ for every } i,j\geq 0,
    \end{equation}
    where $Y/k$ is the special fibre with respect to a smooth proper model of $X$. Our main result is the following. 
\begin{theorem}\label{thm : main}
	Let $\mathcal X$ be a smooth proper model of $X$ with special fibre $Y$. Assume that \ref{eq : equalityhodgenumber} holds.  If $\H^0(X,\Omega^2)\neq 0$, then there exists a finite field extension $L/K$ and $\A\in \Br(X_L)[p]$ such that $\ev^F_{\A}: X(F)\to \Br(F)[p]$ is surjective for every finite field extension $F/L$. 
\end{theorem}
    Since classes $\A\in \Br(X)$ with prime-to-$p$ order have constant evaluation map after a finite extension\footnote{Indeed, for $n$ coprime to $p$, smooth proper base change gives isomorphisms $\H^2(X_{\overline K},\Z/n)\simeq \H^2(\mathcal X_{\mathcal O_{\overline K}},\Z/n)\simeq \H^2(Y_{\overline k},\Z/n)$. Hence, for all $\mathcal A\in \Br(X)[n]$ there exists a finite extension $K\subseteq L$ such that $\mathcal A_{L}$ lifts to $\Br(\mathcal X_{\mathcal O_L})$, so that $\ev^L_{\mathcal A}$ factors through $\Br(\mathcal O_{L})=0$.}, Theorem \ref{thm : main} gives classes of the smallest possible order. The methods used in the proof of Theorem \ref{thm : main} allow us to go beyond merely showing the existence of arithmetically interesting Brauer classes. For products of elliptic curves, we show how to construct such classes, generalising previous constructions in the literature, see Proposition~\ref{prop : elliptic curves} and Corollary~\ref{cor : elliptic}. For abelian varieties of positive $p$-rank and their associated Kummer varieties, we give lower bounds for the number of interesting Brauer classes, see Theorem~\ref{thm : abelian and kumm varieties}.
    \subsubsection{Global applications}
    The main motivations behind Theorem \ref{thm : main} are the following global corollaries. 
    Assume now that $Z$ is smooth proper variety over a number field $M$ such that $\H^0(Z,\Omega^2)\neq 0$. Then, by generic freeness, the base change of $Z$ to the completion at a place satisfies \ref{eq : equalityhodgenumber} for all but finitely many places. Since any $\alpha\in \Br(Z)$ with non-constant evaluation on the points over a completion of $M$ obstructs weak approximation for $M$-rational points (see e.g.\ \cite[Chapter 13]{CoSko}), Theorem \ref{thm : main} implies the following: 
    \begin{corollary}\label{cor : weak}
	Let $Z$ be a smooth proper variety over a number field satisfying weak approximation over all finite extensions. Then $\H^0(Z,\Omega^2)=0$.
\end{corollary}
Corollary~\ref{cor : weak} answers a question of Wittenberg (see \cite[Question 1.4]{BNwild}), which is a special case of his more general question asking whether varieties over number fields that satisfy the Hasse principle and weak approximation over all finite extensions are geometrically rationally connected, hence have vanishing extremal Hodge numbers. 

Theorem \ref{thm : main} actually implies something stronger. Recall that a place $v$ of $M$ is said to be potentially relevant to the Brauer--Manin obstruction to weak approximation on $Z$ if there exist a finite extension $N/M$, a place $v'$ of $N$ lying over $v$, and $\mathcal A\in \Br(Z_N)$ such that $\ev^{N_{v'}}_\A$ is non-constant.
\begin{corollary}\label{cor : potrel}
    	Let $Z$ be a smooth proper variety over a number field $M$ such that $\H^0(Z,\Omega^2)\neq 0$. Then all but finitely many places of $M$ are potentially relevant to the Brauer--Manin obstruction to weak approximation on $Z$.
\end{corollary}
For some important classes of varieties, such as abelian varieties, K3 surfaces, complete intersections in products of projective spaces, assumption \ref{eq : equalityhodgenumber} is satisfied at all places of good reduction. 
\subsubsection{Comparison with previous results}
Since the Brauer classes we construct in Theorem \ref{thm : main} and Corollary \ref{cor : potrel}
have non-constant evaluation over all finite extensions of $K$, they are necessarily of {transcendental} nature, meaning that they do not vanish in $\Br(X_{\bar{K}})$. In general, the construction of transcendental Brauer classes of arithmetic interest is a hard problem. Until very recently, only a handful of explicit examples of transcendental elements obstructing weak approximation had been constructed, see e.g.\ \cite{Hararifirst, Wittenberg, HasVarWAorder2, IeronymouOrd2, PreuOrd3, NewtonCM, Newtoncorrigendum, BergVarOrd3, SkoroIeronymou, ErrataIeronymouSkoro, NewtonMK3}.  

It was only in \cite[Theorem~C]{BNwild} that a first general existence result was shown. There, Bright and the second-named author show that, upon replacing $K$ with a finite field extension, transcendental classes with non-constant evaluation map exist as soon as the reduction 
of the variety has a non-zero global 2-form and is \textit{ordinary} in the sense of Bloch--Kato~\cite{BKpadic}. 
In~\cite{margherita2}, the third-named author gives examples showing how one can construct an element in the Brauer group obstructing weak approximation starting from a non-zero global $2$-form on the special fibre. These include some examples with non-ordinary reduction, indicating that the ordinary condition in~\cite[Theorem~C]{BNwild} is not necessary.

 By way of contrast with the ordinary reduction hypothesis in~\cite[Theorem~C]{BNwild}, recall that if $Z$ is an elliptic curve over $\Q$, then, by \cite{Elkies}, there are infinitely many places of supersingular reduction (of positive density if $Z$ has complex multiplication), while for abelian varieties of dimension $\geq 4$, the existence of primes of good ordinary reduction is still an open question. In the case of products of elliptic curves and abelian varieties, our results  show that there are often more interesting classes in the non-ordinary reduction cases than in the ordinary one. See Section \ref{sec : introexamples} for more details.
\begin{remark}
    In the context of Theorem \ref{thm : main}, if $\H^0(Y,\Omega^2)=0$ then there are no classes whose evaluation map remains non-constant over all finite extensions of $K$, as follows by combining Corollary \ref{cor : Brauervanishing} with \cite[Theorem~8.1 and (8.0.1)]{BKpadic}. Hence,  under assumption \ref{eq : equalityhodgenumber}, the hypothesis  on the existence of global differential $2$-forms is necessary. On the other hand, we do not know whether the assumption \ref{eq : equalityhodgenumber} is really necessary. The first interesting test case is a non-classical Enriques surface over a finite extension of $\Q_2$ with good reduction, for which, to the best of our knowledge, not many results exist in the literature. In this situation, torsion in \'etale cohomology has to be treated with different techniques. We hope to return to this problem in the near future.
\end{remark}
 \subsection{Strategy}
\subsubsection{Cohomological interpretation}
The starting point for the proof of Theorem \ref{thm : main} is a cohomological interpretation of Brauer classes with non-constant evaluation map,  which can be deduced by pushing the techniques in \cite{BNwild}. Consider the following subgroup of $\Br(X_{\overline K})[p]$ of ``geometrically boring'' classes
\begin{align*}
\Br(X_{\overline K})[p]^{\gb}
:=\bigcup_{L/K \textrm{ finite }}\{\A_{\bar{K}}\mid \A\in\Br(X_L)[p] \textrm{ and } \forall F/L \textrm{ finite }  
\ev_{\A}^{F} \textrm{ is constant}\}.
\end{align*}
The conclusion of Theorem \ref{thm : main} can be shown to be equivalent to the statement that 
$\Br(X_{\overline K})[p]^{\gb} \neq \Br(X_{\overline K})[p]$, see Proposition~\ref{prop : gb alt}. Writing $\overline K(X)^{\sh}$ for the strict henselisation of the function field of $X_{\overline K}$ with respect to the $p$-adic valuation, one proves the following: 
\begin{theorem}\label{thm : Brauerandhenselisation}
Assume that $X$ is smooth and proper and has good reduction. Then there is a natural isomorphism of $\Gal(\bar{K}/K)$-modules
$$\Br(X_{\overline K})[p]/\Br(X_{\overline K})[p]^{\gb}=\Image(\H^2(X_{\overline K},\mu_p)\rightarrow \H^2(\overline K(X)^{\sh},\mu_p)).$$	
\end{theorem}
\subsubsection{Prismatic interpretation}
With Theorem \ref{thm : Brauerandhenselisation} in hand, we can apply the machinery of prismatic cohomology. Let $\mathcal O_{\mathbb C_p^{\flat}}$ be the tilt of the ring of integers $\Ocp$ of $\mathbb C_p:=\widehat{\overline \Q}_p$, which is  (non-canonically) isomorphic to the completion of the algebraic closure of the power series ring in one variable over $k$. Our arguments will concern $\H^*(\mathcal X_{\Ocp},\Delta/p)$, the reduction modulo $p$ of the prismatic cohomology theory defined in \cite{BS}. Under assumption \ref{eq : equalityhodgenumber}, it is a finite free $\Ocpflat$-module (Proposition \ref{prop : basicpropertiesprismatic}), endowed with a Frobenius $\varphi: \H^*(\mathcal X_{\Ocp},\Delta/p)\rightarrow \H^*(\mathcal X_{\Ocp},\Delta/p)$.  Let $d\in \mathcal O_{\mathbb C_p^{\flat}}$ be the element defined in \eqref{eq : defofd} and write 
$$\H^n(\mathcal X_{\Ocp},\Delta/p)^{\varphi=d^i}:=\Ker(\H^n(\mathcal X_{\Ocp},\Delta/p)\xrightarrow{\varphi-d^i} \H^n(\mathcal X_{\Ocp},\Delta/p)),$$
for $i,n\in \Z_{\geq 0}$.

Building on the work in \cite{BMS1,BMS2,BS} and, in particular, using the syntomic interpretation of $p$-adic vanishing cycles given in \cite[Section 10]{BMS2}, we prove the following:
\begin{theorem}\label{thm : prismaticinterpretationkernel}
Let $\mathcal X$ be a smooth proper model of $X$ with special fibre $Y$. Assume that \ref{eq : equalityhodgenumber} holds. Then one has 
    $$\Ker(\H^n(X_{\overline K},\mathbb Z/p)\rightarrow \H^n(\overline K(X)^{\sh},\mathbb Z/p))\simeq \H^n(\mathcal X_{\Ocp},\Delta/p)^{\varphi=d^{n-1}} .$$
\end{theorem}
Thanks to the \'etale comparison theorem for prismatic cohomology and to Theorems \ref{thm : Brauerandhenselisation} and \ref{thm : prismaticinterpretationkernel}, Theorem \ref{thm : main} is implied by the following:
\begin{theorem}\label{thm : jumpofdimension}
	Let $\mathcal X$ be a smooth proper model of $X$ with special fibre $Y$. Assume that \ref{eq : equalityhodgenumber} holds. If $\H^i(X,\Omega^{2n-i})\neq 0$ for some $i\neq n$ , then 
	$$\Dim_{\mathbb F_p}(\H^{2n}(\mathcal X_{\Ocp},\Delta/p)^{\varphi=d^n})<\rank_{\mathcal O_{\mathbb C^{\flat}_p}}(\H^{2n}(\mathcal X_{\Ocp},\Delta/p)).$$
	\end{theorem}
       \begin{remark}\label{rmk : everythingisdifficult}
     Observe that Theorems \ref{thm : prismaticinterpretationkernel} and \ref{thm : jumpofdimension} give information about different maps: for an even degree $n=2k$,
     Theorem \ref{thm : prismaticinterpretationkernel} concerns the map $\varphi-d^{n-1}: \H^{n}(\mathcal X_{\Ocp},\Delta/p)\rightarrow \H^n(\mathcal X_{\Ocp},\Delta/p))$, while Theorem \ref{thm : jumpofdimension} applied to $\H^{n}(\mathcal X_{\Ocp},\Delta/p)$ concerns the map $\varphi-d^{k}: \H^{n}(\mathcal X_{\Ocp},\Delta/p)\rightarrow \H^{n}(\mathcal X_{\Ocp},\Delta/p))$. The only cohomology group for which they can be used together is $\H^2(\mathcal X_{\Ocp},\Delta/p)$, because in this case $n=2$ so that $n-1=1=k$. In contrast, in general, the equality 
	$$\Dim_{\mathbb F_p}(\H^{n}(\mathcal X_{\Ocp},\Delta/p)^{\varphi=d^{n-1}})=\rank_{\mathcal O_{\mathbb C^{\flat}_p}}(\H^{n}(\mathcal X_{\Ocp},\Delta/p))$$
    can hold for some $n>2$, even when the non-diagonal Hodge numbers of $X$ are non-zero.
    This can be seen already for $n=4$ and $X$ a triple product of an  elliptic curve, see Example \ref{ex : 3elliptic curves}. 
    \end{remark}
	\subsection{Strategy for the proof of Theorem \ref{thm : jumpofdimension}}
    To simplify the discussion, we assume in this subsection that $n=1$.
One can show that if $\H^0(X,\Omega^{2})\neq 0$ then  $\varphi:\H^{2}(\mathcal X_{\Ocp},\Delta/p)\rightarrow \H^{2}(\mathcal X_{\Ocp},\Delta/p)$ is not zero when tensored with $\overline{k}$. However, this information is not enough to control the dimension of $\H^{2}(\mathcal X_{\Ocp},\Delta/p)^{\varphi=d}$, as Example \ref{ex : slopedfixedpoint} shows. Hence, a more refined control on the action of the Frobenius is needed; we achieve this via a Newton-above-Hodge type of argument. 

Choose a compatible system $\{d^{a}\}_{a\in \Q_{\geq 0}}$ of roots of $d$ in $\Ocpflat$. For $a\in \mathbb Q_{\geq 0}$, let $\mathcal O(-a)$ be the free $\mathcal O_{\mathbb C^{\flat}_p}$-module of rank $1$, with generator $e$, on which there is a semi-linear Frobenius acting as $\varphi(e)=d^ae$. First, we show in Proposition \ref{prop : existence filtration} that $\H^2(\mathcal X_{\Ocp},\Delta/p)$ admits a Frobenius-equivariant decreasing filtration $F^j$ such that $F^{j}/F^{j+1}\simeq \mathcal O(-a_j)$ for some $a_j\in \Q_{\geq 0}$. 
While these $a_j$ might depend on the choice of the filtration (see Example \ref{ex : slope}), we show in Proposition \ref{prop : existence filtration}(2) that their sum does not. Hence, we define the total slope $\TS (\H^2(\mathcal X_{\Ocp},\Delta/p))$ of $\H^2(\mathcal X_{\Ocp},\Delta/p)$ to be this sum. Some (semi-)linear algebra over $\Ocpflat$ (Lemma \ref{lem : d-fixedpoints}) reduces the proof of Theorem \ref{thm : jumpofdimension} to showing the equality 
$$\TS (\H^{2}(\mathcal X_{\Ocp},\Delta/p))=\rank(\H^{2}(\mathcal X_{\Ocp},\Delta/p)).$$
The following theorem, whose proof is inspired by~\cite{NewtonaboveHodge}, proves a more general result, which should be thought of as a ``Newton-above-Hodge" statement and may be of independent interest. 
\begin{theorem}\label{thm : HodgevsNewton}
Let $\mathcal X$ be a smooth proper model of $X$ with special fibre $Y$ and let $h^{i,j}=\Dim_{\mathbb C_p}(\H^j(X,\Omega^i_{X/\mathbb C_p}))$. Assume that \ref{eq : equalityhodgenumber} holds. 
Then $$\TS (\H^n(\mathcal X_{\Ocp},\Delta/p))=\sum^n_{i=0}i\cdot h^{i,n-i}.$$
	\end{theorem}
\subsection{Special cases}\label{sec : introexamples}
We make Theorem \ref{thm : main} more explicit for certain families of abelian varieties and their associated Kummer varieties. For abelian varieties, one can use  prismatic Dieudonn\'e theory, from \cite{Dieudonneprismatic}, to make prismatic cohomology more concrete. 
\subsubsection{Products of elliptic curves}
By combining Theorems~\ref{thm : Brauerandhenselisation} and \ref{thm : prismaticinterpretationkernel} with prismatic Dieudonn\'e theory,  we prove the following: 
\begin{proposition}\label{prop : elliptic curves}
Let
$X=Z\times W$ for elliptic curves  $Z,W$  with good reduction and  N\'eron models $\mathcal Z$,$\mathcal W$. Then there exists a natural isomorphism
$$\Br(X_{\overline K})[p]/\Br(X_{\overline K})[p]^{\gb}\simeq \Hom_{\mathbb C_p}
   (Z[p],W[p])/\Hom_{\Ocp}(\mathcal Z [p],\mathcal{W}[p]).$$
	\end{proposition}
	Abusing notation, we write $\Hom_{\overline K}(Z,W)$ for its image in $\Hom_{\overline K}(Z[p],W[p])$. Recalling from~\cite[Proposition~3.3]{SZtorsionEC} that $\Br(X_{\overline K})[p]\simeq \Hom_{\overline K}(Z[p],W[p])/\Hom_{\overline K}(Z,W),$
Proposition \ref{prop : elliptic curves} shows that an element in $\Br(X_{\overline K})[p]$ lies in $\Br(X_{\overline K})[p]^{\gb}$ if and only if a corresponding element in $\Hom_{\overline K}(Z[p],W[p])$ lifts to the integral model over $\Ocp$. Since $X$ has trivial canonical bundle, combining this with Proposition \ref{prop : braueruseful} shows that a Brauer class in $\Br(X)[p]$ associated to a homomorphism $\sigma: Z[p]\rightarrow W[p]$ defined over $K$ has non-constant evaluation map (over $K$) as soon as $\sigma$ does not lift to the integral models over $\Ocp$.

This gives a criterion for having non-constant evaluation map that can be applied to concrete Brauer classes. For example, when $Z=W$ is a CM elliptic curve, we study the Brauer class associated to the action of complex conjugation and show how to reinterpret and generalise examples constructed in \cite{SkoroIeronymou, NewtonCM, NewtonMK3}, see Sections~\ref{subsubsec : CM}--\ref{subsubsec : compare literature}. In this setting, we show in Corollary~\ref{cor : CM} that in most cases of supersingular reduction one has $\Br(X_{\overline K})[p]^{\gb}=0$, which cannot happen in the ordinary reduction case (cf.\ Corollary~\ref{cor : elliptic}).
\subsubsection{Abelian varieties of positive $p$-rank and associated Kummer varieties}
Now suppose that $X$ is an abelian variety. Since $\H^{2}(\mathcal X_{\Ocp},\Delta/p)\simeq \Lambda^2 \H^{1}(\mathcal X_{\Ocp},\Delta/p)$, prismatic Dieudonn\'e theory can be used once again to get more explicit results. We prove a lower bound for the quotient
$\Br(X_{\overline K})[p]/\Br(X_{\overline K})[p]^{\gb}$. Moreover, we show that this bound can be transferred to the Kummer variety associated to an $X[2]$-torsor $T$ over $K$.
\begin{theorem}\label{thm : abelian and kumm varieties}
Let $X$ be an abelian variety of dimension $g\geq 2$ with good reduction. Assume that $X$ has a polarisation of degree prime to $p$ and that the special fibre has $p$-rank $e>0$. Then
\[
    \Dim_{\mathbb F_p}(\Br(X_{\overline K})[p]/\Br(X_{\overline K})[p]^{\gb})\geq 2g-1-e.
\]
Furthermore, if $p$ is odd, then for any $X[2]$-torsor $T$ over $K$, 
\[
    \Dim_{\mathbb F_p}(\Br(\mathrm{Kum}(X_T)_{\overline K})[p]/\Br(\mathrm{Kum}(X_T)_{\overline K})[p]^{\gb})\geq 2g-1-e.
\] 
where $\mathrm{Kum}(X_T)$ is the Kummer variety associated to $T$. 
\end{theorem}

The bounds given in Theorem~\ref{thm : abelian and kumm varieties} are attained when $X$ is a product of two elliptic curves, see Corollary~\ref{cor : elliptic}. 
 \begin{remark}
 Some of the results in this article admit slight generalisations with similar proofs, at the expense of adding more notation and introducing more cohomology theories. For example, Theorem \ref{thm : Brauerandhenselisation} holds with $p$ replaced by $p^n$ for any $n\in\Z_{>0}$, as can be seen using truncated de Rham--Witt complexes, while Proposition \ref{prop : elliptic curves} can be extended to general products satisfying \eqref{eq : equalityhodgenumber}, using integral prismatic cohomology and the results in \cite{uinfinitytorsion}. Since our main purpose was to prove Theorem \ref{thm : main} and its Corollaries \ref{cor : weak}, \ref{cor : potrel}, for which $p$-torsion classes give the best possible result, we have chosen to keep the notation, the machinery and the proofs as minimal as possible to arrive there.
    \end{remark}
\subsection{Organisation of the paper}
In Section \ref{sec : brauerclasses}, we study the relationship between Brauer classes with non-constant evaluation, henselisation and $p$-adic vanishing cycles. In Section \ref{sec : prismatic}, after recalling some preliminaries on prismatic cohomology, we relate $p$-adic vanishing cycles (hence Brauer classes with non-constant evaluation) to prismatic cohomology. In Section \ref{sec : totalslope} we prove a Newton-versus-Hodge type of theorem for modulo $p$ prismatic cohomology. In Section \ref{sec : prooffinal} we put everything together to prove the main Theorem \ref{thm : main}. Finally, in Section \ref{sec : abelian varieties} we use the techniques of the preceding sections to study products of elliptic curves and abelian varieties, where the results are more explicit.
\subsection{Acknowledgements}
This research was partly supported by the grant ANR–23–CE40–0011 of Agence Nationale de la Recherche. Rachel Newton was supported by UKRI Future Leaders Fellowship MR/T041609/1 and MR/T041609/2.
Part of this work was done during the special trimester \emph{Arithmetic geometry of K3 surfaces} at the Bernoulli Center for Fundamental Studies. We thank the organisers, as well as Stefan Schr\"oer for a very helpful discussion on a filtration on the $p$-torsion of elliptic curves, and Livia Grammatica for a very nice talk on the Nygaard filtration in positive characteristic. We also thank Arthur-César Le Bras and Matthew Morrow for useful email exchanges, and  Giuseppe Ancona, Martin Bright and Olivier Wittenberg for useful comments on a first version of this article.  

\subsection{Conventions}
If $L$ is a valued field, we write $\mathcal O_L$ for the valuation ring, $\mathfrak m_L\subseteq \mathcal O_L$ for the maximal ideal, and $k_L$ for the residue field. If $\mathcal O_L$ is a discrete valuation ring, we write $\pi_L$ for a uniformiser and we omit the subscript $_L$ if it is clear from the context. If $K$ is a $p$-adic field, we write $\mathbb C_p$ for the completion of an algebraic closure of $K$. 

If $R\rightarrow S$ is a morphism of rings and $X$ is a scheme over $R$, we write $X_S:=X\times_RS$. If $X$ is a scheme and $A$ a ring, we write $\Sh( X, A)$ for the category of \'etale sheaves on $X$ with coefficients in $A$ and $D^{b}(X,A)$ for its bounded derived category. We omit $A$ if $A=\Z$. If $\mathcal F^{\bullet}$ is in $D^{b}(X,A)$ and $n\in \mathbb Z$, we write $\tau_{\leq n}\mathcal F^{\bullet}$ and $\tau_{>n}\mathcal F^{\bullet}$ for the canonical truncations, so that there is an exact triangle
$$\tau_{\leq n}\mathcal F^{\bullet}\rightarrow \mathcal F^{\bullet}\rightarrow \tau_{>n}\mathcal F^{\bullet}.$$
We let $\mathcal H^i(\mathcal F^{\bullet})$ be the $i^{th}$ cohomology sheaf of $\mathcal F^{\bullet}$ and $\H^i(X,\mathcal F^{\bullet})$ be the hypercohomology of $\mathcal F^{\bullet}$. We omit $X$ if it is clear from the context. 
We write $\mathcal F^{\bullet}/p:=\mathcal F^{\bullet}\otimes^L \Z/p$
for the derived tensor product with $\Z/p$. This gives an exact functor
$$(-)\otimes^L \Z/p :D^{b}(X,A)\rightarrow D^{b}(X,A/p).$$
Most of the time, when talking about global differential forms on varieties over a field, we will suppress from the notation the subscript keeping track of the variety itself. This will not be the case in Section~\ref{sec : prismatic}, where we work with differential forms and the de Rham complex on schemes over a given base ring. 
\section{Brauer classes with non-constant evaluation map}\label{sec : brauerclasses}
In this section, we prove Theorem \ref{thm : Brauerandhenselisation} and its Corollary \ref{cor : Brauervanishing} that gives a cohomological interpretation of geometrically boring Brauer classes. In order to do this, we start in Section \ref{subsec : filtrationbrauer} by recalling some results and notation from \cite{Kato, BNwild} about residues and refined Swan conductors. In Section \ref{subsec : proofofbrauerhens}, we use this theory to prove Theorem \ref{thm : Brauerandhenselisation}, from which we deduce Corollary \ref{cor : Brauervanishing}. 
\subsection{Filtration on the Brauer group and refined swan conductor}\label{subsec : filtrationbrauer}
Let $K$ be a finite extension of $\Q_p$ with residue field $k$. Fix a uniformiser $\pi$ of $\mathcal O_K$. Let $\mathcal X/\mathcal O_K$ be a smooth proper scheme with special fibre $Y/k$ and generic fibre $X/K$. Let $K(X)$ be the function field of $X$, let $K(X)^{\mathrm h}$ be its henselisation  with respect to the $p$-adic place, and let $K(X)^{\sh}$ be its strict henselisation. 
Let $g_K:\Br(X)[p]\rightarrow \Br(K(X)^{\sh})[p]$ be the restriction map. 

From \cite{BNwild}, we learn that by pulling back Kato's \cite{Kato} Swan conductor filtration defined on $\H^2(K(X)^h,\Z/p(1))= \Br(K(X)^h)[p]$ to $\Br(X)[p]$, we get a finite exhaustive increasing filtration
$$0\subseteq \fil_0\Br(X)[p]\subseteq \dots  \subseteq \fil_i\Br(X)[p]=\Br(X)[p].$$
This filtration is such that
\begin{equation}\label{eq: fil0 as kernel to strict hens}
    \fil_0\Br(X)[p]=\Ker(\Br(X)[p]\rightarrow \Br(K(X)^{\sh})[p])
\end{equation}
and there are maps, realised as pullbacks of the analogous maps defined by Kato \cite{Kato} at the level of the henselian field $K(X)^h$, 
$$\partial: \fil_0\Br(X)[p]\rightarrow \H^1(Y,\Z/p)\  \text{ and }  \ \text{rsw}_{n,\pi}\colon \fil_n\Br(X)[p]\rightarrow \H^0(Y,\Omega^1)\oplus \H^0(Y,\Omega^2), \text{ for }n\geq 1,$$
 such that $\fil_{n-1}=\ker(\text{rsw}_{n,\pi})$. 
In \cite{BNwild} Bright and Newton show that these maps control arithmetic properties of the elements in the Brauer group. In particular, they satisfy the following properties.
\begin{theorem}\cite[Corollary 3.7 and Theorem B]{BNwild}\label{thm : summaryBN} \leavevmode
    \begin{enumerate}
        \item The kernel of the residue map $\partial$ coincides with $\Br(\mathcal{X})[p]$;
        \item Let $n\geq 1$ and $\A\in \fil_n\Br(X)[p]$ with $\mathrm{rsw}_{n,\pi}(\A)=(\alpha,\beta)\in \H^0(Y,\Omega^2)\oplus \H^0(Y,\Omega^1)$. If there exists $P_0\in Y(k)$ such that $(\alpha_{P_0},\beta_{P_0})\ne 0$, then $\mathrm{ev}_{\A}\colon X(K)\rightarrow \Br(K)[p]$ is surjective.
    \end{enumerate}
\end{theorem}

In order to prove Theorem \ref{thm : Brauerandhenselisation}, we need to give more details on the construction of the residue map $\partial$ at the level of the henselian field $K(X)^h$. 

\vspace{1mm}
\subsubsection{The map  $\lambda_\pi$}
Following \cite{Kato},  we write
$$\Z/p(1):=\Omega^{1}_{\log}[-1] \text{ in } D^b(k(Y)) \quad \text{and} \quad \Z/p(1)\text{ in } D^b(K(X)^h)$$
for the ($1$-shifted) sheaf of $1$-logarithmic differential forms and the Tate twist of the constant sheaf, respectively.

In \cite[(1.4)]{Kato} Kato builds lifting maps 
\[
\iota^1 \colon \H^1(k(Y),\Z/p)\rightarrow \H^1(K(X)^h,\Z/p)\quad \text{and}\quad \iota^2 \colon \H^2(k(Y),\Z/p(1))\rightarrow \H^2(K(X)^h,\Z/p(1)).
\] 
Using the Kummer map $(K(X)^h)^\times \rightarrow \H^1(K(X)^h, \Z/p(1))$ and the cup product, we define a product
\begin{align*}
        \H^1(K(X)^h,\Z/p)\times (K(X)^h)^\times &\rightarrow \H^{2}(K(X)^h,\Z/p(1))\\
        (\chi,a) &\mapsto \{\chi, a\}.
\end{align*}
By  \cite[(1.4)]{Kato}, the map 
\begin{align*}
    \lambda_\pi\colon \H^2(k(Y),\Z/p(1))\oplus \H^{1}(k(Y),\Z/p)&\rightarrow \H^2(K(X)^h,\Z/p(1))\\
    (\chi,\psi)&\mapsto \iota^2(\chi)+\{\iota^{1}(\psi),\pi\}
\end{align*}
is injective. 
\subsubsection{Definition of $\fil_0$ and $\partial$}\label{sec : def0}
The subgroup $\fil_0\Br(K(X)^h)[p]\subseteq \Br(K(X)^h)[p]= \H^2(K(X)^h,\Z/p(1))$ is then defined\footnote{Kato's original definition of $\fil_0$ is different; however, he proves in \cite[Proposition 6.1(1)]{Kato} that the definition we give here is equivalent to his. In \emph{loc.\ cit.}, he also proves that $\fil_0\Br(K(X)^h)[p]$ can be realised as the kernel of the natural map $\H^2(K(X)^h,\Z/p(1))\rightarrow \H^2(K(X)^{sh},\Z/p(1))$, which gives the description of $\fil_0$ provided in equation~\eqref{eq: fil0 as kernel to strict hens}.} as 
\[  
   \fil_0\Br(K(X)^h)[p]:=\Image\big (\lambda_\pi\colon \H^2(k(Y),\Z/p(1))\oplus \H^1(k(Y),\Z/p)\rightarrow \H^2(K(X)^h,\Z/p(1))\big ).
\]
Since  $\lambda_\pi$ is injective, we can then define the residue map
\begin{equation}\label{eq : defres}
 \partial:=p_2\circ \lambda_\pi^{-1} \colon \fil_0 \Br(K(X)^h)[p]\rightarrow \H^1(k(Y),\Z/p)
    \end{equation}
    which is the projection onto $\H^1(k(Y),\Z/p)$ of the inverse of $\lambda_\pi$. Its restriction to $\fil_0 \Br(X)[p]\subseteq \Br(X)[p]$ induces (see~\cite[Proposition 3.1]{BNwild}) the required map
\[
    \partial\colon\fil_0 \Br(X)[p] \rightarrow \H^1(Y,\Z/p).
\]
\subsection{Proof of Theorem \ref{thm : Brauerandhenselisation}}\label{subsec : proofofbrauerhens}
Since elements of $\NS(X_{\overline K})$ are already zero when restricted to $\overline K(X)$,
the restriction map
$$\H^2(X_{\overline K},\Z/p(1))\rightarrow \H^2(\overline K(X)^{\sh},\Z/p(1))$$
factors through a map
$$g:\Br(X_{\overline K})[p]\rightarrow \H^2(\overline K(X)^{\sh},\Z/p(1)).$$

Theorem~\ref{thm : Brauerandhenselisation} now follows from the equality of sets (1) and (2) in the next proposition, while the equality of sets (2) and (3) shows that the conclusion of Theorem \ref{thm : main} is equivalent to the statement that $\Br(X_{\overline K})[p]^{\gb}\neq \Br(X_{\overline K})[p]$.
\begin{proposition}\label{prop : gb alt}
    The following sets are equal:
    \begin{enumerate}
        \item $\Ker(g)$;
        \item $\Br(X_{\overline K})[p]^{\gb}$;
        \item $\bigcup_{L/K \textrm{ finite }}\{\A_{\bar{K}}\mid \A\in\Br(X_L)[p] \textrm{ and } \forall F/L \textrm{ finite } \exists F'/F \textrm{ finite with } 
|\ev_{\A}^{F'}(X(F'))|<p \}$.
    \end{enumerate}
\end{proposition}

The inclusion $(2)\subseteq (3)$ follows immediately from the definition of $\Br(X_{\overline K})[p]^{\gb}$. We will prove $(1)\subseteq (2)$ and $(3)\subseteq (1)$.
\subsubsection{$(1)\subseteq(2)$}
First we prove the inclusion $\Ker(g)\subseteq \Br(X_{\overline K})[p]^{\gb}$. Let $\widetilde \A\in \Ker(g)$ and let $L$ be a finite extension of $K$ such that there exists $\A\in \Br(X_L)[p]$ with $\A_{\overline K}=\widetilde \A$. Since $\widetilde \A\in \Ker(g)$, upon replacing $L$ with a finite extension, we can assume that $\A$ is in the kernel of $g_L\colon \Br(X_L)[p]\rightarrow \Br(L(X)^{\sh})[p]$, i.e.\ in $\fil_0\Br(X_L)[p]$, cf.~\eqref{eq: fil0 as kernel to strict hens}. By Proposition \ref{prop:geomfil0} below, upon replacing $L$ with a finite extension, we can assume that $\partial(\A)=0$. By Theorem~\ref{thm : summaryBN}(1), this implies that $\A\in \Br(\mathcal X_L)[p]\subseteq \Br(X_L)[p]$, hence $\A_F\in \Br(\mathcal X_F)[p]\subseteq \Br(X_F)[p]$ for all finite extensions $F/L$. But then, since $\mathcal X(\mathcal O_F)=X(F)$, the evaluation map factors through $\Br(\Os_F)=0$ and hence is constantly zero.

\begin{proposition}\label{prop:geomfil0}
Let $\A\in\fil_0\Br (X)[p]$. Then the image of $\A$ in $\Br (K(X)^{\mathrm h})[p]$ can be written as a sum $\B+\C$ where $\B\in\fil_0\Br(K(X)^{\mathrm h})[p]$ has residue zero and $\C\in \Br (K(X)^{\mathrm h})[p]$ is such that $\C_{L(X)^{\mathrm h}}=0$ for some finite extension $L/K$. 
\end{proposition}
\begin{proof}
By definition of $\fil_0\Br (X)[p]$ (see Section ~\ref{sec : def0}), there exist $(\chi,\psi)\in \H^2(k(Y),\Z/p(1))\oplus \H^1(k(Y),\Z/p)$ such that 
\[\A_{K(X)^{\mathrm h}}=\lambda_{\pi}(\chi,\psi).\]
Let $\B=\lambda_{\pi}(\chi,0)\in \fil_0\Br(K(X)^{\mathrm h})[p]$. Note that $\B$ has residue zero by definition  \eqref{eq : defres}, hence it remains to prove that $\C:=\A_{K(X)^{\mathrm h}}-\B$ vanishes after a finite extension. This follows from the equalities 
$$\C_{K(X)^{\mathrm h}}=\lambda_{\pi}(\chi,\psi)-\lambda_{\pi}(\chi,0)=\lambda_{\pi}(0,\psi)=\{\iota^1(\psi),\pi \}$$
since $\{\iota^1(\psi),\pi \}$ is split by adjoining a root of $\pi$.
\end{proof}
\subsubsection{$(3)\subseteq(1)$}
Let $\widetilde \A\in \Br(X_{\overline K})[p]\setminus \Ker(g)$ and let $L/K$ be a finite extension such that there exists $\A\in \Br(X_L)[p]$ with $\A_{\overline K}=\widetilde \A$. Since $\widetilde \A \not \in \Ker(g)$, for every finite extension $F/L$ one has $g_F(\A)\neq 0$ where  $g_F:\Br(X_F)[p]\rightarrow \Br(F(X)^{\sh})[p]$ is the natural map, i.e. $\A\not \in \fil_0\Br(X_F)[p]$. Upon replacing $L$ with a finite extension, Proposition \ref{prop : braueruseful} below shows that the evaluation map for $\widetilde\A$ takes $p$ values, as required. This completes the proof of Proposition~\ref{prop : gb alt}, and hence the proof of Theorem~\ref{thm : Brauerandhenselisation}. 
\begin{proposition}\label{prop : braueruseful}
	Let $\mathcal{A}\in \Br(X)[p]$ and assume that $\A \not \in \fil_0\Br(X_F)[p]$ for all finite extensions $F$ of $K$.
	Let $k'/k$ be a finite extension such that, for $i=1,2$, for every non-zero $\gamma\in \H^0(Y,\Omega^i)$ there exists $P\in Y(k')$ such that $\gamma_P\neq 0$.
    Let   $K'/K$ be the corresponding unramified extension. Then, for every finite extension $L/ K'$ the evaluation map $\ev_{\A}^L:X(L)\rightarrow \Br(L)[p]$ is surjective.
	 \end{proposition}
	 \begin{proof}
Since $\A_L \notin \fil_0\Br(X_L)[p]$,  there exists $n\geq 1$ such that $\A_L\in \fil_n\Br(X_L)[p]\setminus \fil_{n-1}\Br(X_L)[p]$. Therefore, $\rsw_{n,\pi_L}(\mathcal A_L)=(\alpha,\beta)\neq (0,0)$, where $(\alpha,\beta)\in \H^0(Y_{k_L},\Omega^2)\oplus \H^0(Y_{k_L},\Omega^1)$. Since $k_L/k$ 
is Galois, Lemma~\ref{lem : vanishing of sections} below shows that there exists $P\in Y(k_L)$ such that $(\alpha_P,\beta_P)\neq (0,0)$. Now the result follows from Theorem \ref{thm : summaryBN}(2).
\end{proof}
\begin{lemma}\label{lem : vanishing of sections}
Let $Y$ be a smooth proper variety over a  field $k$ and $i\in \mathbb Z_{\geq 0}$. Let $\ell/k$ be a finite Galois extension. Assume that for every non-zero $\alpha\in \H^0(Y,\Omega^i)$ there exists a point $P\in Y(\ell)$ such that $\alpha_P\neq 0$. Then, for every non-zero $\beta\in \H^0(Y_{\ell},\Omega^i)$, there exists a point $P\in Y(\ell)$ such that $\beta_P\neq 0$.
\end{lemma}
\proof
Let $\alpha_1,\dots,\alpha_n$ be a $k$-basis of $\H^0(Y,\Omega^i)$.
By flat base change, any $\beta\in \H^0(Y_{\ell},\Omega^i)$ can be written as
\[\beta=\sum_{j=1}^nb_j\alpha_j\]
for some $b_j\in\ell$. Suppose that $\beta_P=0$ for all $P\in Y(\ell)$. For $\lambda\in\ell$ and $\sigma\in\Gal(\ell/k)$, 
\[\sigma_*(\lambda\beta)=\sum_{j=1}^n\sigma(\lambda b_j)\alpha_j\]
and $\sigma_*(\lambda\beta)_P=0$ for all $P\in Y(\ell)$. Hence,
\[\sum_{j=1}^n\Tr_{\ell/k}(\lambda b_j)\alpha_j\in \H^0(Y,\Omega^i)\]
vanishes at all $P\in Y(\ell)$. By our assumption, this implies that $\Tr_{\ell/k}(\lambda b_j)=0$ for all $j$. Since $\lambda\in \ell$ was arbitrary and $\ell/k$ is separable, this implies that $b_j=0$ for all $j$, by non-degeneracy of the trace form. 
\endproof
\subsection{$p$-adic vanishing cycles}
    To be able to reinterpret Theorem \ref{thm : Brauerandhenselisation} in terms of prismatic cohomology, we need first to reinterpret it in terms of the $p$-adic vanishing cycles spectral sequence, following \cite[Proof of Theorem C]{BNwild}. Let $j:X\rightarrow \mathcal X$ be the natural open immersion. 
Choosing an isomorphism between $\mu_p$ and $\Z/p$ over $\bar{K}$, we may and do replace $\mu_p$ with $\Z/p$ in the statement of Theorem \ref{thm : Brauerandhenselisation}. 
Recall that there is a spectral sequence
$$E^{a,b}_{2}:=\H^a(\mathcal X_{\mathcal O_{\overline K}},R^bj_*\Z/p)\Rightarrow \H^{a+b}(X_{\overline K},\Z/p).$$
 We consider the the edge map
$$\H^n(X_{\overline K},\Z/p)\rightarrow  \H^0(\mathcal X_{\mathcal O_{\overline K}},R^nj_*\Z/p)$$
and its variant over $\Ocp$
$$\H^n(X_{\mathbb C_p},\Z/p)\rightarrow  \H^0(\mathcal X_{\Ocp},R^nj_*\Z/p).$$

\begin{corollary}\label{cor : Brauervanishing}
One has an equality
    $$\Ker(\H^n(X_{\overline K},\Z/p)\rightarrow \H^n(\overline K(X)^{\sh},\Z/p))=\Ker(\H^n(X_{\mathbb C_p},\Z/p)\rightarrow  \H^0(\mathcal X_{\Ocp},R^nj_*\Z/p)),$$
    and a group isomorphism  $$\Br(X_{\bar{K}})[p]/\Br(X_{\bar{K}})[p]^{\gb}=\Image(\H^2(X_{\mathbb C_p},\Z/p)\rightarrow  \H^0(\mathcal X_{\Ocp},R^2j_*\Z/p)).$$
\end{corollary}
\proof
The second part follows from the first part and Theorem \ref{thm : Brauerandhenselisation}. For the first part, write $\overline K(X)^{\mathrm h}$ for the henselisation of $\overline K(X)$, so that $\overline K(X)^{\sh}$ is the maximal unramified extension of $\overline K(X)^{\mathrm h}$. Let $\Gamma:=\mathrm{Gal}(\overline K(X)^{\sh}/\overline K(X)^{\mathrm h})$. 
Consider the Hochschild--Serre spectral sequence for $\overline K(X)^{\mathrm h}\subseteq \overline K(X)^{\sh},$
$$E^{a,b}_{2}:=\H^a(\Gamma,\H^b(\overline K(X)^{\sh},\Z/p))\Rightarrow \H^{a+b}(\overline K(X)^{\mathrm h},\Z/p).$$

The edge maps
$$\H^n(X_{\overline K},\Z/p)\rightarrow \H^0(\mathcal X_{\mathcal O_{\overline K}},R^nj_*\Z/p)\quad \text{and}\quad \H^{n}(\overline K(X)^{\mathrm h},\Z/p)\rightarrow \H^0(\Gamma,\H^n(\overline K(X)^{\sh},\Z/p))$$
fit into a commutative diagram
\begin{center}
	\begin{tikzcd}
\H^n(X_{\overline K},\Z/p)\arrow{d}\arrow{r}& \H^{n}(\overline K(X)^{\mathrm h},\Z/p)\arrow{d}\\
\H^0(\mathcal X_{\mathcal O_{\overline K}},R^nj_*\Z/p)\arrow{r}{g}&\H^0(\Gamma,\H^n(\overline K(X)^{\sh},\Z/p))\arrow[hook]{r}&\H^n(\overline K(X)^{\sh},\Z/p),
	\end{tikzcd}
\end{center}
in which the horizontal arrows in the square are the natural restriction maps and the bottom right map is the natural inclusion. 
By \cite[Lemma 3.4]{BNwild}, the map $g$ is injective, hence 
$$\Ker(\H^n(X_{\overline K},\Z/p)\rightarrow \H^n(\overline K(X)^{\sh},\Z/p))=\Ker(\H^n(X_{\overline K},\Z/p)\rightarrow  \H^0(\mathcal X_{\mathcal O_{\overline K}},R^nj_*\Z/p)).$$
Since \'etale cohomology is invariant under algebraically closed field extensions, one has
$$\Ker(\H^n(X_{\overline K},\Z/p)\rightarrow  \H^0(\mathcal X_{\mathcal O_{\overline K}},R^nj_*\Z/p))\simeq \Ker(\H^n(X_{\mathbb C_p},\Z/p)\rightarrow  \H^0(\mathcal X_{\Ocp},R^nj_*\Z/p)),$$
and this concludes the proof.\endproof
\section{Prismatic cohomology and $p$-adic vanishing cycles}\label{sec : prismatic}
In this section we prove Theorem \ref{thm : prismaticinterpretationkernel}. We start in Section \ref{subsec : review} with some background on $\Ocpflat$, $\Ainf$ and prismatic cohomology from \cite{BS}. With this in hand, in Section \ref{subsec : prismaticcomputation} we perform some computations in the mod-$p$ version of prismatic cohomology that we will use in the rest of the paper. Finally, in Section \ref{subsec : prismaticvanishing} we prove Theorem \ref{thm : prismaticinterpretationkernel}. 

In this section $k$ denotes the residue field of $\Ocp$. Let $\mathcal X$ be a smooth proper scheme over $\Ocp$ and $\widehat{\mathcal X}$ its formal $p$-adic completion. We write $Y$ for $\mathcal X_k$ and $X$ for $\mathcal X_{\mathbb C_p}$.
Let 
$$h^{a,b}:=\Dim_{\mathbb C_p}(\H^b(X,\Omega^a_{X/\mathbb C_p}))\quad \text{and} \quad h^n:=\Dim_{\mathbb C_p}(\H^n_{\dr}(X/\mathbb C_p))=\Dim_{\mathbb Q_p}(\H^n_{\et}(X,\Qp)).$$

\subsection{Recollection on prismatic cohomology}\label{subsec : review}
\subsubsection{Tilt}
A reference for what follows is  \cite[Sections 2 and 3]{Bhattperfectoid}.
Let $\mathcal{O}_{\mathbb C^{\flat}_p}$ be the tilt of $\mathcal{O}_{\mathbb C_p}$. Recall that 
$$\mathcal{O}_{\mathbb C^{\flat}_p}\simeq \varprojlim_{\varphi_{\Ocp/p}}\Ocp/p\simeq \varprojlim_{x\mapsto x^p} \Ocp,$$
where the first isomorphism is as rings and the second as multiplicative monoids, and the first limit is done along the power of the absolute Frobenius $\varphi_{\Ocp/p}:\Ocp/p\rightarrow \Ocp/p$ and the second along the $p^{th}$-power map $(-)^p:\Ocp\rightarrow \Ocp$.

The choice of a compatible system $\{\zeta_{p^{n}}\}$ of primitive $p$-power roots of unity in $\mathcal{O}_{\mathbb C_p}$ gives, via the isomorphism $\mathcal{O}_{\mathbb C^{\flat}_p}\simeq  \varprojlim_{x\mapsto x^p} \Ocp$, an element 
$$\epsilon:=(1,\zeta_p,\dots, \zeta_{p^n},\dots )\in \mathcal{O}_{\mathbb C^{\flat}_p}.$$
The $p$-adic valuation of $\Ocp$ induces, by precomposing with the projection on the first component,  a rank $1$ valuation $v$ on $\Ocpflat$, such that the element 
\begin{equation}\label{eq : defofd}
    d:=\sum^{p-1}_{i=0}\epsilon^{i/p},
\end{equation}
has $v(d)=1$. Then $v$ makes $\Ocpflat$ a complete valuation ring, whose maximal ideal $\mathfrak m^{\flat}$ is generated by $d^{\alpha}$ for $\alpha\in \Q_{> 0}$. 

\subsubsection{$\Ainf$}
A reference for what follows is \cite[Section 3, Example 3.16]{BMS1}. 
Set $\Ainf:=W(\mathcal{O}_{\mathbb C^{\flat}_p})$ and write $\varphi_{\Ainf}:\Ainf\rightarrow \Ainf$ for the Frobenius automorphism. If it is clear from the context, we avoid writing the subscript $_{\Ainf}$. 
If $x\in \mathcal{O}_{\mathbb C^{\flat}_p}$, we denote by $[x]$ its Teichm\"{u}ller lift. Recall that there is a natural map $\theta:\Ainf\rightarrow \mathcal O_{\mathbb C_p}$, whose kernel is principal, generated by  
$$\xi:=\sum^{p-1}_{i=0}[\epsilon^{i/p}]$$
whose reduction modulo $p$ is $d$.

  Since $\Ker(\theta)=(\xi)$ is principal, for every $n\geq 1$, sending $1$ to $1\otimes\xi^n$ and $1\otimes d^n$ induces natural isomorphisms
  $$\Ainf\rightarrow \Ainf \otimes (\xi^n) \quad \text{and}\quad \Ocpflat\rightarrow \Ocpflat \otimes (d^n),$$
  respectively.
  If $M$ is an $\Ainf$-module (resp.\ $\Ocpflat$-module), we will identify $M\otimes (\xi^n)$ (resp.\ $M\otimes (d^n)$) with $M$ along this isomorphism, and if $f:N\rightarrow M\otimes (\xi^n)$ (resp.\ $f:N\rightarrow M\otimes (d^n)$) is a morphism of 
$\Ainf$-modules (resp.\ $\Ocpflat$-modules), we will 
write $f_n:N\rightarrow M$ for the induced morphism. 

We will work mainly with the following commutative diagram
\begin{center}
	\begin{tikzcd}
&  \mathcal O_{\mathbb C_p}\\
\Ainf\arrow{d}{\beta}\arrow{ru}{\widetilde{\theta}}\arrow{r}{\varphi}  &\Ainf\arrow{d}{\beta}\arrow{u}{\theta}\\
W(k)\arrow{r}{\varphi} &W(k),
	\end{tikzcd}
\end{center}
where $\beta$ is induced by the projection $\mathcal O_{\mathbb C_p^{\flat}}\rightarrow k$ via the Witt vectors functoriality and $\widetilde{\theta}:=\theta\circ\varphi^{-1}$, and its modulo $p$  version
\begin{center}
	\begin{tikzcd}
&  \mathcal O_{\mathbb C_p}/p\simeq\mathcal O_{\mathbb C^{\flat}_p}/d\\
\mathcal O_{\mathbb C^{\flat}_p}\arrow{d}{\beta}\arrow{ru}{\widetilde{\theta}}\arrow{r}{\varphi}  &\mathcal O_{\mathbb C^{\flat}_p}\arrow{d}{\beta}\arrow{u}{\theta}\\
k\arrow{r}{\varphi} &k.
	\end{tikzcd}
\end{center}
\subsubsection{Some sheaves on the prismatic sites }
A reference for what follows is \cite[Section 4]{BS}. Consider the prism $\Delta:=(\Ainf,(\xi))$. 
Let $\Sh_{\Delta}(\mathcal X)$ 
be the category of sheaves on the prismatic site of $\widehat{\mathcal X}$ (with the flat topology, see \cite[Section 4.1]{BS}). There are two interesting sheaves of $\Ainf$-modules: 
\begin{enumerate}
	\item $\mathcal O_{\mathcal X/\Delta}$, which sends $(B,J)$ to $B$
	\item $\overline{\mathcal O}_{\mathcal X/\Delta}$, which sends $(B,J)$ to $B/J$. 
\end{enumerate} 
Let $\Sh_{\et}(\widehat{\mathcal X})$ be the category of sheaves on the \'etale site of $\widehat{\mathcal X}$.  
There is a natural functor (see \cite[Construction 4.4]{BS})
$$v_*:\Sh_{\et}(\widehat{\mathcal X})\rightarrow \Sh_{\Delta}(\mathcal X).$$
Let
$$\Delta:=Rv_*\mathcal O_{\mathcal X/\Delta}\quad \text{and}\quad \overline{\Delta}:=Rv_*\overline{\mathcal O}_{\mathcal X/\Delta}.$$
These are complexes of sheaves of $\Ainf$ and $\Ocp$-modules, respectively, such that
$$\Delta\otimes^L \Ocp\simeq \overline{\Delta}$$
and $\Delta$ is endowed with a natural $\varphi_{\Ainf}$-linear Frobenius $\varphi: \Delta\rightarrow \Delta.$

Set $$\Delta^{(1)}:=\varphi_{\Ainf}^*\Delta,$$
so that the $\varphi_{\Ainf}$-linear Frobenius $\varphi: \Delta\rightarrow \Delta$ induces an $A_{\inf}$-linear map
$\phi:\Delta^{(1)}\rightarrow \Delta$ and a $\varphi_{\Ainf}$-linear  map $\varphi^{(1)}:\Delta^{(1)}\rightarrow \Delta^{(1)}$, fitting into a commutative diagram
\begin{center}
	\begin{tikzcd}
\Delta\arrow{r}{\varphi} \arrow[swap]{d}{s}&\Delta\arrow{d}{s}\\
\Delta^{(1)} \arrow{ru}{\phi}\arrow[swap]{r}{\varphi^{(1)}}& \Delta^{(1)} 
	\end{tikzcd}
\end{center}
where $s$ is the natural inclusion.

By \cite[Theorems 15.2 and 15.3]{BS} (see also \cite[
Construction 5.26]{logprism2} for a global statement in a more general situation)
there exists a decreasing filtration, called the  Nygaard filtration, 
$$\iota:N^{\geq i}\rightarrow \Delta^{(1)}$$ 
and an $\Ainf$-linear map $\phi_i:N^{\geq i}\rightarrow \Delta$, called the $i^{th}$-divided Frobenius, making the following diagram commutative
\begin{center}
	\begin{tikzcd}
N^{\geq i}\arrow{r}{\phi_i}\arrow[bend left]{rr}{\phi \circ \iota }& \Delta\arrow{r}{\xi^i}&  \Delta.
	\end{tikzcd}
\end{center}
We write $\varphi^{(1)}_i:N^{\geq i}\rightarrow \Delta^{(1)}$ for the $\varphi_{\Ainf}$-linear compositum $s\circ\phi_i$.
\subsubsection{Comparison isomorphisms}
The following theorem summarises the main results we need from \cite{BS}. 
\begin{theorem}\label{thm : BS}
	There are exact triangles in $D^b(\widehat{\mathcal X},\Ocpflat)$ 	
    \begin{equation}\label{eq : hodgeTate}
	\tau_{\leq i-1}(\overline{\Delta}/p)\rightarrow \tau_{\leq i} (\overline{\Delta}/p)\rightarrow \Omega^i_{\mathcal X/p/(\Ocp/p)} 
\end{equation}  
	\begin{equation}\label{eq : grnygard}
N^{\geq i+1}/p\rightarrow N^{\geq i}/p\xrightarrow{\phi_i}\tau_{\leq i}(\overline\Delta/p)
\end{equation}  
and an isomorphism
\begin{equation}\label{eq : derhamcomparison}
	\varphi_{\Ainf}^* \overline{\Delta}/p \simeq \Omega^{\bullet}_{\mathcal \mathcal X/p/(\Ocp/p)}.
\end{equation}
\end{theorem}
\proof
Since $\mathcal X$ is smooth, $\Omega^j_{\mathcal X/\mathcal O_{\mathbb C_p}}$ is locally free. Hence 
(\ref{eq : hodgeTate}) follows from applying $-\otimes^L\Z/p$ to the Hodge--Tate comparison from \cite[Theorem 4.11]{BS}, thanks to Lemma \ref{lem : spectraldegeneration}(1) below. Similarly, (\ref{eq : grnygard}) follows from \cite[Theorem 15.3]{BS} (see also \cite[
Corollary 5.27]{logprism2} for a global statement) after applying $-\otimes^L \Z/p$, thanks to Lemma \ref{lem : spectraldegeneration}(1) below. Finally, (\ref{eq : derhamcomparison}) follows from the de Rham comparison \cite[Theorem 6.4]{BS} after applying $-\otimes^L \Z/p$.
\endproof
\begin{lemma}\label{lem : spectraldegeneration}
	Let $R$ be a ring, let $\mathcal F^{\bullet}\in D^{b}(\widehat{\mathcal X},R)$, and assume that $\mathcal H^n(\mathcal F^\bullet)$ is finite locally free over $R$ for every $n$. 
	\begin{enumerate}
		\item For every ring morphism $R\rightarrow S$, there is a natural isomorphism $(\tau_{\leq i}\mathcal F^{\bullet})\otimes^L S\simeq \tau_{\leq i}(\mathcal F^{\bullet}\otimes^L S)$.
		\item Assume further that $R$ is a coherent complete local ring  with residue field $k$ and that $\H^{i}(\mathcal H^n(\mathcal F^\bullet))$ and $\H^{i}(\mathcal F^\bullet)$ are finite free for every $i$ and $n$. 
Suppose that the conjugate spectral sequence 
		$$_kE_{2}^{a,b}:=\H^{a}(\mathcal H^b(\mathcal F^{\bullet}\otimes^Lk))\Rightarrow \H^{a+b}(\mathcal F^{\bullet}\otimes ^Lk)$$ 	
		for $\mathcal F^{\bullet}\otimes^L k$ degenerates at the second page. Then the conjugate spectral sequence 
		$$ E_{2}^{a,b}:=\H^{a}(\mathcal H^b(\mathcal F^{\bullet}))\Rightarrow \H^{a+b}(\mathcal F^{\bullet})$$
				for $\mathcal F^{\bullet}$ degenerates at the second page. 
	\end{enumerate}
    
\end{lemma}
\proof
\begin{enumerate}
\item[]
	\item Assume that $\mathcal F^{\bullet}$ is concentrated in degree $\leq r$. We prove the statement by decreasing induction on $i$, the case $i=r+1$ being clear since both sides identify with  $\mathcal F^{\bullet}\otimes^L S$. Assume now $i<r+1$ and consider the exact triangle
	 $$\tau_{\leq i}\mathcal{F}^{\bullet}\rightarrow \tau_{\leq i+1}\mathcal{F}^{\bullet}\rightarrow \mathcal H^{i+1}(\mathcal{F}^{\bullet}),$$
	 giving an exact triangle
	 	 $$(\tau_{\leq i}\mathcal{F}^{\bullet})\otimes^L S\rightarrow (\tau_{\leq i+1}\mathcal{F}^{\bullet})\otimes^L S\rightarrow (\mathcal H^{i+1}(\mathcal{F}^{\bullet}))\otimes^LS.$$
	 By induction $(\tau_{\leq i+1}\mathcal{F}^{\bullet})\otimes^L S\simeq \tau_{\leq i+1}(\mathcal{F}^{\bullet}\otimes^L S)$. Since $(\tau_{\leq i}\mathcal{F}^{\bullet})\otimes^L S$ is concentrated in degree $\leq i$, we get a natural commutative diagram with exact rows
	 	\begin{center}
	\begin{tikzcd}
(\tau_{\leq i}\mathcal{F}^{\bullet})\otimes^L S\arrow{r}\arrow{d}&  (\tau_{\leq i+1}\mathcal{F}^{\bullet})\otimes^LS\arrow{r} \arrow{d}&  (\mathcal H^{i+1}(\mathcal{F}^{\bullet}))\otimes^L S\arrow{d}\\
\tau_{\leq i}(\mathcal{F}^{\bullet}\otimes^L S)\arrow{r}&  \tau_{\leq i+1}(\mathcal{F}^{\bullet}\otimes^LS) \arrow{r} \arrow{r}&  \mathcal H^{i+1}(\mathcal{F}^{\bullet}\otimes^L S)
	\end{tikzcd}
\end{center}
Since $\mathcal H^{i+1}(\mathcal{F}^{\bullet})$ is flat, the right vertical map is an isomorphism. By the induction hypothesis, the middle vertical map is an isomorphism. This implies the statement. 
\item Observe that by assumption and point (1) the natural map
\begin{equation}\label{eq : compatibility spectral sequence}
	E^{a,b}_{2}\otimes k=\H^{a}(\mathcal H^b(\mathcal F^{\bullet}))\otimes k\simeq \H^{a}(\mathcal H^b(\mathcal F^{\bullet}\otimes^L k))= \text{ }_kE^{a,b}_{2}
\end{equation}
is an isomorphism. 
Let $F^{i}_n$ be the $i^{th}$ piece of the filtration on $\H^n(\mathcal F^{\bullet})$ induced by the spectral sequence. 
We prove by induction on $i\geq 0$ that, for every $n$, $E^{i,n-i}_{2}=E^{i,n-i}_{\infty}$ 
and $F^{i+1}_n$ is finite locally free. 

If $i=0$, the equality $E^{0,n}_{2}=E^{0,n}_{\infty}$ follows from the fact that the edge map $$\H^n(\mathcal F^{\bullet})\rightarrow \H^0(\mathcal H^n(\mathcal F^{\bullet}))$$ 
is surjective, by (\ref{eq : compatibility spectral sequence}), Nakayama's Lemma and the degeneration of the spectral sequence $_kE_{2}^{a,b}$. Then there is an exact sequence
$$0\rightarrow F^1_{n}\rightarrow{} \H^n(\mathcal F^{\bullet})\rightarrow E^{0,n}_{2}\rightarrow 0$$
showing that $F^1_{n}$ is finite locally free, since $E^{0,n}_{2}$ and $\H^n(\mathcal F^{\bullet})$ are.  

Now assume that $i>0$. By the induction hypothesis, $E^{j,n-j}_{2}=E^{j,n-j}_{\infty}$ 
for $j<i$ and every $n$. In particular, all the morphisms with target $E^{i,n-i}_{a}$ vanish for every $a\geq 2$, so that $E^{i,n-i}_{\infty}\subseteq E_2^{i,n-i}$ and there is a map $F^i_n\rightarrow E_2^{i,n-i}$. Again by Nakayama's Lemma, the degeneration of the spectral sequence $_kE_{2}^{a,b}$ and (\ref{eq : compatibility spectral sequence}), this map is surjective. Hence  $E_2^{i,n-i}= E_{\infty}^{i,n-i}$ and there is a short exact sequence 
$$0\rightarrow F^{i+1}_{n}\rightarrow F^{i}_{n} \rightarrow E^{i,n-i}_{2}\rightarrow 0,$$
so that $F^{i+1}_{n}$ is finite and locally free, since $E^{i,n-i}_{2}$ and $F^{i}_n$ are. This concludes the proof. 
\qedhere  \end{enumerate}
\endproof
\subsection{Cohomological computations}\label{subsec : prismaticcomputation}
In this section we carry out the main computations with prismatic cohomology that we will need in the rest of the paper. Assumption \ref{eq : equalityhodgenumber} gives a way to compute the Hodge groups, which we then use to study prismatic cohomology and the cohomology of the Nygaard filtration. 
\subsubsection{Hodge cohomology groups}
\begin{lemma}\label{lem : Hodgegroups}
	Assume that \ref{eq : equalityhodgenumber} holds. Let $R$ be an $\mathcal O_{\mathbb C_p}$-algebra. 
	\begin{enumerate}
		\item The natural map 
		$$\H^i(\mathcal X,\Omega^j_{\mathcal X/\mathcal O_{\mathbb C_p}})\otimes R\rightarrow \H^i(\mathcal X_R,\Omega^j_{\mathcal X_R/R})$$ is an isomorphism and $\H^i(\mathcal X_R,\Omega^j_{\mathcal X_R/R})$ is a free $R$-module of rank $h^{j,i}$.
		\item The Hodge-to-de Rham spectral sequence 
		 $$E_1^{a,b}:=\H^b(\mathcal X_R,\Omega^a_{\mathcal X_R/R})\Rightarrow \H^{a+b}_{\dR}(\mathcal X_R/R)$$
		 degenerates at the first page.
		 \item The natural map 
		$$\H^{n}_{\dR}(\mathcal X/{\mathcal O_{\mathbb C_p}})\otimes R\rightarrow \H^n_{\dR}(\mathcal X_R/R)$$ is an isomorphism and $\H^n_{\dR}(\mathcal X_R/R)$ is a free $R$-module of rank $h^n$.
\item If $R$ is a complete coherent local ring of characteristic $p$ with residue field k, then the conjugate spectral sequence
 $$E_2^{a,b}:=\H^a(\mathcal X_R,\mathcal H^b(\Omega^{\bullet}_{\mathcal X_R/R}))\Rightarrow \H^{a+b}_{\dR}(\mathcal X_R/R)$$
degenerates at the second page and $\H^a(\mathcal X_R,\tau_{\leq n}\Omega^{\bullet}_{\mathcal X_R/R})$
is a locally free $R$-module of rank $\sum^n_{i=0}h^{i,a-i}$.
	\end{enumerate}
\end{lemma}
\proof
\begin{enumerate}
\item[]
	\item Since $\Ocp$ is reduced, this follows from Grauert's theorem (see e.g.\ \cite[25.1.5, p.~731]{Vakil} and \cite[Exercise 25.2.C (c), p.~740]{Vakil} for the extension to the non-Noetherian situation), thanks to the assumption \ref{eq : equalityhodgenumber}.
	\item By point (1) and the functoriality of the Hodge-to-de Rham spectral sequence, it is enough to show this for $R=\Ocp$. In this case, it follows from the commutative diagram
	\begin{center}
	\begin{tikzcd}
\H^b(\mathcal X,\Omega^{a}_{\mathcal X/\Ocp})\arrow{r}{d}\arrow{d} & \H^b(\mathcal X,\Omega^{a+1}_{\mathcal X/\Ocp})\arrow{d}\\
\H^b(X,\Omega^{a}_{X/\mathbb C_p})\arrow{r}{d}& \H^b(X,\Omega^{a+1}_{X/\mathbb C_p}),
	\end{tikzcd}
\end{center}
since the spectral sequence always degenerates in characteristic 0 and the vertical maps are injective by point (1). 
\item Since the Hodge-to-de Rham spectral sequence degenerates at the first page by the previous point, the induced filtration $F^i$ on $\H_{\dR}^{n}(\mathcal X/\Ocp)$ satisfies 
$$F^{i}/F^{i+1}\simeq \H^{n-i}(\mathcal X,\Omega^{i}_{\mathcal X/\Ocp}),$$
in particular $F^{i}/F^{i+1}$ is free.  Hence, one has a  morphism of spectral sequences
	\begin{center}
	\begin{tikzcd}
\H^b(\mathcal X,\Omega^a_{\mathcal X/\Ocp})\otimes R\arrow{d}\arrow[Rightarrow]{r}& \H^{a+b}_{\dR}(\mathcal X/\Ocp)\otimes R\arrow{d}\\
\H^b(\mathcal X_R,\Omega^a_{\mathcal X_R/R})\arrow[Rightarrow]{r}& \H^{a+b}_{\dR}(\mathcal X_R/R)
	\end{tikzcd}
\end{center}
and the conclusion follows from point (1). 
\item By Cartier, $F_{X/R,*}\mathcal H^i(\Omega_{\mathcal X_R/R}^{\bullet})\simeq \Omega^i_{\mathcal X^{(1)}_R/R}$, where $X^{(1)}$ is the Frobenius twist of $X$ over $R$ and $F_{X/R}:X\rightarrow X^{(1)}$ is the relative Frobenius. In particular, 
$$\mathcal H^i(\Omega^{\bullet}_{\mathcal X_R/R})\quad \text{and} \quad \H^i(\mathcal X_R,\mathcal H^n(\Omega_{\mathcal X_R/R}^{\bullet}))$$
are $R$-flat, the second by point (1). 
Hence, by Lemma~\ref{lem : spectraldegeneration}, it remains to show that  
 $$\H^a(Y,\mathcal H^b(\Omega^{\bullet}_{Y/k}))\Rightarrow \H^{a+b}_{\dR}(Y/k)$$
 degenerates at the second page. This follows again from Cartier's isomorphism and point (2), since $Y$ is proper over $k$. \qedhere
\end{enumerate}
\endproof
\subsubsection{Prismatic cohomology groups}
\begin{proposition}\label{prop : basicpropertiesprismatic}
Assume that \ref{eq : equalityhodgenumber} holds. 
\begin{enumerate}
\item $\H^n(\Delta^{(1)}/p)$ and $\H^n(\Delta/p)$ are free $\mathcal O_{\mathbb C_p^{\flat}}$-modules of rank equal to $h^n$.
\item The natural map $$\H^n(\Delta^{(1)}/p)\otimes k\rightarrow \H^n_{\dr}(Y/k)$$ is an isomorphism. 
	\item The natural maps  
	$$\H^n(\Delta^{(1)}/p)/d\rightarrow \H^n(\overline \Delta^{(1)}/p) \quad \text{and}\quad \H^n(\Delta/p)/d\rightarrow  \H^n(\overline \Delta/p)$$ are isomorphisms. In particular, $\H^n(\overline \Delta^{(1)}/p)$ and $\H^n(\overline \Delta/p)$ are free $\mathcal O_{\mathbb C_p^{\flat}}/d$-modules of rank equal to $h^n$.
	\item 	The conjugate spectral sequence for $\overline \Delta/p$
	$$E_2^{a,b}:=\H^a(\mathcal H^b(\overline \Delta/p))\Rightarrow \H^{a+b}(\overline \Delta/p)$$
   degenerates at the second page.
	\end{enumerate}
\end{proposition}
\proof
\begin{enumerate}
	\item[]
	\item Since the Frobenius $\mathcal O_{\mathbb C_p^{\flat}}\rightarrow \mathcal O_{\mathbb C_p^{\flat}}$ is an isomorphism, by \cite[Corollary 4.12]{BS} one has that  $\H^n(\Delta^{(1)}/p)\simeq \H^n(\Delta/p)\otimes_{\varphi_{\Ainf}}\mathcal O_{\mathbb C_p^{\flat}}$, so that it is enough to prove the statement for $\Delta^{(1)}$. Since $\mathcal O_{\mathbb C_{p}^{\flat}}$ is a valuation ring with maximal ideal $(d^{1/p^{\infty}})$, it is enough to show that the the minimal number of generators of $\H^n(\Delta^{(1)}/p)/d$ is at most the dimension of $\H^n(\Delta^{(1)}/p)[1/d]$ as a $\mathbb C_p^{\flat}$-vector space. By the \'etale comparison \cite[Theorem 1.8]{BS}, one has
	$\H^n(\Delta^{(1)}/p)[1/d]\simeq \H^n_{\et}(X,\Z/p)\otimes \mathbb C_p^{\flat}$, see for example \cite[Lemmas 2.1.6 and 2.2.10]{FKW}. Hence, it suffices to show that the minimal number of generators of $\H^n(\Delta^{(1)}/p)/d$ is at most $\Dim_{\F_p}(\H^n_{\et}(X,\Z/p))$. By \eqref{eq : derhamcomparison}, there is an isomorphism
	$\H^n(\Delta^{(1)}/p\otimes^L \Ocpflat/d)\simeq \H^n_{\dR}(\mathcal X/p/ (\Ocpflat/d)),$
	hence an inclusion 
	$$\H^n(\Delta^{(1)}/p)/d\hookrightarrow \H^n_{\dR}(\mathcal X/p/(\Ocpflat/d)).$$
By Lemma \ref{lem : Hodgegroups}(3), it remains to show that  $\Dim_{\mathbb{C}_p}(\H^n_{\dR}(X/\mathbb C_p)\leq \Dim_{\F_p}(\H^n_{\et}(X,\Z/p))$, which follows from the fact that  
$\Dim_{\mathbb{C}_p}(\H^n_{\dR}(X/\mathbb C_p))=\Dim_{\mathbb{Q}_p}(\H^n_{\et}(X,\Q_p))$ and the universal coefficients short exact sequence
$$0\rightarrow \H^n_{\et}(X,\Z_p)/p\rightarrow \H^n_{\et}(X,\Z/p)\rightarrow \H^{n+1}_{\et}(X,\Z_p)[p]\rightarrow 0.$$
\item This follows from (1) and Theorem \ref{thm : BS}.
\item This follows from (1) and the universal coefficient theorem. 
\item Since $\mathcal X/p$ is smooth over $\Ocp/p$, the sheaf $\Omega^{i}_{\mathcal X/p/(\Ocp/p)}$ is finite locally free over $\Ocp/p$. 
By Lemma \ref{lem : Hodgegroups} and point (3), $\H^i\bigl(\mathcal X/p,\Omega^j_{\mathcal X/p/(\Ocp/p)}\bigr)$ and $\H^i(\overline \Delta/p)$  are finite free $\Ocpflat/d$-modules. 
Hence, Lemma \ref{lem : spectraldegeneration}(2) and Theorem \ref{thm : BS} imply that it is enough to show that the conjugate spectral sequence 
		$$E_{2}^{a,b}:=\H^{a}(\mathcal H^b(\overline \Delta\otimes^Lk))\Rightarrow \H^{a+b}(\overline \Delta\otimes ^Lk)$$ 	
		for $\overline \Delta\otimes^L k$ degenerates at the second page. Since these are finite-dimensional $k$-vector spaces, it is enough to show that the conjugate spectral sequence 
			$$E_{2}^{a,b}:=\H^{a}(\mathcal H^b(\varphi^*(\overline \Delta\otimes^Lk)))\Rightarrow \H^{a+b}(\varphi^*(\overline \Delta\otimes ^Lk))$$ 
			degenerates at the second page. 
			By Theorem \ref{thm : BS} and Lemma \ref{lem : Hodgegroups}, the latter is identified with the conjugate spectral sequence for de Rham cohomology of $Y$, which degenerates at the second page, again by Lemma \ref{lem : Hodgegroups}. \qedhere
    \end{enumerate}
    \endproof
\subsubsection{Nygaard filtration cohomology groups}    \begin{proposition}\label{prop : nygaard}
    Assume that \ref{eq : equalityhodgenumber} holds. Then
        the natural maps $\H^n(N^{\geq i}/p)\xrightarrow{\iota} \H^n(\Delta^{(1)}/p)$ are injective with image
	$\phi^{-1}(d^i\H^n(\Delta/p))$. In particular, $\H^n(N^{\geq i}/p)$ is a free $\mathcal O_{\mathbb C_p^{\flat}}$-module of rank equal to $h^n$.	
    \end{proposition}
    \proof
		The proof is by induction on $i$, where the case $i=0$ follows from the fact that $N^{\geq 0}=\Delta^{(1)}$. Assume that $i\geq 0$ and that the statement holds for $j<i+1$. Since 
        $$\Image(\H^n(N^{\geq i+1}/p)\xrightarrow{\iota} \H^n(\Delta^{(1)}/p))\subseteq \phi^{-1}(d^{i+1}\H^n(\Delta/p)),$$
         there is a commutative diagram with exact rows
	\begin{center}
	\begin{tikzcd}[column sep=tiny]
& \H^n(N^{\geq i+1}/p)\arrow{r}{\iota}\arrow{d}& \H^n(\Delta^{(1)}/p)\arrow{r}\arrow[equal]{d}& \H^n((\Delta^{(1)}/p)/(N^{\geq i+1}/p))\arrow{d}{g}\\
0\arrow{r}& \phi^{-1}(d^{i+1}\H^n(\Delta/p))\arrow{r}&\H^n(\Delta^{(1)}/p)\arrow{r}& \H^n(\Delta^{(1)}/p)/\phi^{-1}(d^{i+1}\H^n(\Delta/p)) \arrow{r}&0,
	\end{tikzcd}
\end{center}
in which $g$ is surjective. Hence, the statement is equivalent to the injectivity of $g:\H^n((\Delta^{(1)}/p)/(N^{\geq i+1}/p))\rightarrow  \H^n(\Delta^{(1)}/p))/\phi^{-1}(d^{i+1}H^n(\Delta/p))$ for every $n$ (and similarly for the induction hypothesis).

Observe that $\phi: \H^n(\Delta^{(1)}/p)\rightarrow \H^n(\Delta/p)$ induces a commutative diagram
	\begin{center}
	\begin{tikzcd}[column sep=tiny]
\H^n((\Delta^{(1)}/p)/(N^{\geq i+1}/p))\arrow{rr}{g}\arrow{dr}{\phi}&&\H^n(\Delta^{(1)}/p)/\phi^{-1}(d^{i+1}\H^n(\Delta/p)) \arrow[hook]{dl}{\phi}\\
& \H^n(\Delta/p)/d^{i+1},
	\end{tikzcd}
\end{center}
where the map $\H^n(\Delta^{(1)}/p)/\phi^{-1}(d^{i+1}\H^n(\Delta^{(1)}/p))\rightarrow \H^n(\Delta/p)/d^{i+1}$ is injective.
Therefore, the statement is equivalent to the injectivity of $\phi:\H^n((\Delta^{(1)}/p)/(N^{\geq i+1}/p))\rightarrow  \H^n(\Delta/p)/d^{i+1}$ for every $n$ (and similarly for the induction hypothesis).

Consider the commutative diagram with exact rows
$$\begin{tikzcd}[column sep=tiny]
 \arrow[eq=d]{d} &\H^n((N^{\geq i}/p)/(N^{\geq i+1}/p))\arrow{d}{\phi_i}\arrow{r}&\H^n((\Delta^{(1)}/p)/(N^{\geq i+1}/p))\arrow{d}{\phi}\arrow{r} & \H^n((\Delta^{(1)}/p)/(N^{\geq i}/p))\arrow{d}{\phi}\\
0\arrow{r}&\H^n(\Delta/p)/d\simeq \H^n(\overline \Delta/p) \arrow{r}{d^i}&\H^n(\Delta/p)/d^{i+1}\arrow{r}&\H^n(\Delta/p)/d^{i}\arrow{r}&0. 
	\end{tikzcd}
$$
where the isomorphism in the first term of the bottom row follows from Proposition \ref{prop : basicpropertiesprismatic}. By the induction hypothesis, the rightmost vertical arrow is injective, hence it is enough to show that the map  $\H^n((N^{\geq i}/p)/(N^{\geq i+1}/p))\rightarrow \H^n(\overline \Delta/p)$ is injective. This is induced by the map of sheaves
$$\phi_i:(N^{\geq i}/p)/(N^{\geq i+1}/p)\rightarrow \overline \Delta/p,$$
which, by Theorem~\ref{thm : BS}, induces an isomorphism $(N^{\geq i}/p)/(N^{\geq i+1}/p)\rightarrow \tau_{\leq i}\overline \Delta/p$. Hence, it is enough to show that the natural map 
$$\H^n(\tau_{\leq i}\overline \Delta/p)\rightarrow \H^n(\overline \Delta/p)$$
is injective, which again follows from Proposition \ref{prop : basicpropertiesprismatic}.
\endproof
\subsection{Prismatic interpretation of $p$-adic vanishing cycles}\label{subsec : prismaticvanishing}
\subsubsection{Vanishing cycles and $d^i$-fixed points of Frobenius}
The following is essentially \cite[Theorem 10.1]{BMS2}, as explained in \cite[Remark 3.2]{Morrowpadicvanishing}. It gives a relation between the action of Frobenius in prismatic cohomology, the Nygaard filtration and the sheaves of vanishing cycles. 
\begin{theorem}\label{thm : exactsequencevanishing}
If $j:X\rightarrow \mathcal X$ is the natural inclusion, then for every $i\geq 0$ there is a natural long exact sequence
$$\dots\rightarrow \H^n(\mathcal X, \tau_{\leq i} Rj_*\mathbb Z/p)\rightarrow \H^n(N^{\geq i}/p)\xrightarrow{\iota-\varphi^{(1)}_i} \H^n(\Delta^{(1)}/p)\rightarrow \H^{n+1}(\tau_{\leq i} Rj_*\mathbb Z/p)\rightarrow \dots$$
\end{theorem}
\proof
This follows combining \cite[Theorem 10.1, Remark 10.3]{BMS2} and \cite[Theorem 17.2]{BS}.  More precisely, since $\mathcal X$ is proper, for every torsion \'etale complex of sheaves $\mathcal F$ over $\mathcal X$, the natural maps
$$\H^n(\mathcal X, \mathcal F)\rightarrow \H^n(\widehat{\mathcal X}, i^*\mathcal F)\rightarrow \H^n(Y, i^*\mathcal F)$$
are isomorphisms, where we write $i:Y\rightarrow \mathcal X$ and $i:\widehat{\mathcal X}\rightarrow \mathcal X$ for both the natural morphisms. Hence, by \cite[Theorem 3.5.13, p.~207]{Huberbook}, there is a natural isomorphism
$$\H^n(\mathcal X, \tau_{\leq i} Rj_*\mathbb Z/p)\simeq \H^n(\widehat{\mathcal X}, \tau_{\leq i} Rb_*\mathbb Z/p),$$
where $b:\Sh_{\et}(\widehat{\mathcal X}_{\eta})\rightarrow  \Sh_{\et}(\widehat{\mathcal X})$ is the functor constructed in \cite[(3.5.12), p.~207]{Huberbook} and $\widehat{\mathcal X}_{\eta}$ is the rigid generic fibre of $\widehat{\mathcal X}$. Then the conclusion follows from \cite[Theorem 10.1]{BMS2} and \cite[Theorem 17.2]{BS}, since those give an exact triangle 
$$\tau_{\leq i} Rb_*\mathbb Z/p \rightarrow  N^{\geq i}/p\xrightarrow{\iota-\varphi^{(1)}_i}\Delta^{(1)}/p. 
$$
\qedhere
\endproof
We now explain when the long exact sequence in Theorem \ref{thm : exactsequencevanishing} can be broken into smaller pieces.  
\begin{lemma}\label{lem : injectivityprismatic}
	Assume that $n\leq i+1$. Then the natural map
	$$\H^n(\tau_{\leq i}Rj_*\mathbb Z/p) \rightarrow \H^n(N^{\geq i}/p)$$
	is injective. 
\end{lemma}
\proof
By Theorem \ref{thm : exactsequencevanishing}, it is enough to show that 
$$\iota-\varphi^{(1)}_i: \H^{n-1}(N^{\geq i}/p)\rightarrow \H^{n-1}(\Delta^{(1)})$$
is surjective. 
By \cite[Lemma 5.32]{uinfinitytorsion}, it is enough to show that $\Coker\bigl(\iota-\varphi^{(1)}_i\bigr)$ is finite. 
Again by Theorem \ref{thm : exactsequencevanishing}, it is enough to show that 
$\H^{n}(\tau_{\leq i}Rj_*\mathbb Z/p)$ is finite. 
Since $\tau_{>i}Rj_*\mathbb Z/p$ is concentrated in degrees $\geq i+1$ one has $\H^{k}(\tau_{>i}Rj_*\mathbb Z/p)=0$ for $k\leq i$. 
Hence, since $n\leq i+1$, the exact triangle
$$\tau_{\leq i}Rj_*\mathbb Z/p\rightarrow Rj_*\mathbb Z/p\rightarrow \tau_{>i}Rj_*\mathbb Z/p$$
shows that the natural map
$$\H^{n}(\tau_{\leq i}Rj_*\mathbb Z/p)\rightarrow \H^{n}(Rj_*\mathbb Z/p)$$
is injective (and even an isomorphism if $n\leq i$). This concludes the proof since 
$$\H^{n}(Rj_*\mathbb Z/p)\simeq \H^{n}_{\et}(X,\Z/p)$$
is a finite dimensional $\F_p$-vector space.
\endproof
Observe that the natural map
$$s:\H^n(\Delta/p)\rightarrow \H^n(\Delta^{(1)}/p)$$
is an isomorphism, hence we can build the commutative diagram
\begin{equation}\label{eq : 1-psi1=d-psi}
	\begin{tikzcd}[column sep= small]
0\arrow{r}& \H^n(\tau_{\leq i}Rj_*\mathbb Z/p) \arrow{d}\arrow{r} & \H^n(N^{\geq i}/p)\arrow{d}{s^{-1}\circ \iota}
\arrow{r}{\iota-\varphi^{(1)}_i}& \H^n(\Delta^{(1)}/p)\arrow{d}{ d^i\circ s^{-1}}\\
0\arrow{r}& \H^n(\Delta/p)^{\varphi=d^i}\arrow{r} & \H^n(\Delta/p) 
\arrow{r}{d^i-\varphi}& \H^n(\Delta/p),
	\end{tikzcd}
\end{equation}
where the left vertical map is induced by the other two vertical ones. 
\begin{lemma}\label{lem : vanishingasd-fixedpoints}
Assume that \ref{eq : equalityhodgenumber} holds.
	Then, for $n\leq i+1$, the natural map
	$$\H^n(\tau_{\leq i}Rj_*\mathbb Z/p) \rightarrow H^n(\Delta/p)^{\varphi=d^i}$$
	is an isomorphism.
\end{lemma}
\proof
By Lemma \ref{lem : injectivityprismatic} and Theorem \ref{thm : exactsequencevanishing}, the rows of the diagram (\ref{eq : 1-psi1=d-psi}) are exact. By Propositions \ref{prop : basicpropertiesprismatic} and \ref{prop : nygaard}, the two right vertical arrows are injective. So the conclusion follows from the fact that 
$$\H^n(N^{\geq i}/p)=\phi^{-1}(d^i\H^n(\Delta/p)),$$
again by Proposition \ref{prop : nygaard}.  
\endproof
\subsubsection{Proof of Theorem \ref{thm : prismaticinterpretationkernel}}
By construction, the edge map 
$$\H^n(X_{\mathbb C_p},\Z/p)\rightarrow  \H^0(\mathcal X_{\Ocp},R^nj_*\Z/p),$$
in the Leray spectral sequence is induced by the exact triangle 
$$\tau_{\leq n-1}(Rj_*\Z/p)\rightarrow Rj_*\Z/p\rightarrow \tau_{>n-1}(Rj_*\Z/p),$$
observing that $\H^n(\mathcal X_{\Ocp},\tau_{> {n-1}}(Rj_*\Z/p))\subseteq \H^0(\mathcal X_{\Ocp}, R^nj_*\Z/p).$ 
Hence, 
$$\Ker(\H^n(X_{\mathbb C_p},\Z/p)\rightarrow  \H^0(\mathcal X_{\Ocp},R^nj_*\Z/p))=\H^n(\mathcal X_{\Ocp},\tau_{\leq {n-1}}(Rj_*\Z/p)),$$
and the conclusion follows from Corollary~\ref{cor : Brauervanishing} and Lemma~\ref{lem : vanishingasd-fixedpoints}. 
\begin{example}\label{ex : fixed elliptic curve}
    Assume that $X$ is an elliptic curve. Then Lemma \ref{lem : vanishingasd-fixedpoints} for $i=n=1$ implies that 
    $$\H^1(\Delta/p)^{\varphi=d}\simeq \H^1(\tau_{\leq 1}Rj_*\mathbb Z/p)\simeq \H^1(X,\Z/p)\simeq (\Z/p)^2$$
    has $\mathbb F_p$-dimension equal to 2. 
\end{example}
\section{Total Newton slope and Hodge polygons}\label{sec : totalslope}
In this section we prove Theorem \ref{thm : HodgevsNewton}. We start in Section \ref{sec: Newton and Hodge} by defining and studying the total Newton slope and the Hodge polygon for abstract finite free modules over $\Ocpflat$ endowed with a Frobenius, taking inspiration from \cite{Katzslope}. Almost by construction, the total Newton slope will be at least the height of the Hodge polygon; in fact, in Lemma~\ref{lem : hodgevsnewton} we show that they are equal. Then in Section \ref{subsec : geometric Hodge polygon} we specialise to the geometric setting and control this height via the Hodge numbers of $X$, proving Proposition~\ref{prop : hodgevsgeometrichodge}, which in turn implies Theorem \ref{thm : HodgevsNewton}. 
\subsection{Total Newton slope and Hodge polygon}\label{sec: Newton and Hodge}
For this section, let $M$ be a finite free $\mathcal O_{{\mathbb{C}_p^{\flat}}}$-module of rank $r\in \Z_{\geq 1}$ equipped with a Frobenius semi-linear map $\varphi$ that becomes an isomorphism after inverting $d$. If $M$ and $N$ are such modules, we write $\Hom_{\varphi}(M,N)$ for the $\mathbb F_p$ vector space of morphisms as $\mathcal O_{{\mathbb{C}_p^{\flat}}}$-modules that commute with the Frobenius.

Recall that, by \cite[Lemma 4.13, p.~128]{Milne}, for every integer $i$, 
$$M^{\varphi=d^i}:=\Ker(\varphi-d^i:M\rightarrow M)$$
is an $\mathbb F_p$-vector space of rank at most $r$. If $\mathcal B:={e_1,\dots, e_r}$ is a basis of $M$, we write $M_{\mathcal B}(\varphi)\in M_{r}(\Ocpflat)$ for the matrix whose $i^{th}$-column is the coordinates vector of $\varphi(e_i)$ with respect to $\mathcal B$. Furthermore, we choose a compatible system $\{d^{1/r}\}_{r\in \mathbb N}$ of roots of $d$.
\subsubsection{Total Newton slope}\label{subsec : totalnewton}
\begin{definition}\label{def : twist}
	Given a rational number $\alpha\geq 0$, we denote by $\mathcal O (-\alpha)$ the rank one $\Os_{\mathbb{C}_p^\flat}$-module 
		\[
		\Os(-\alpha)=\Os_{\mathbb{C}_p^\flat}\cdot e
		\]
		with Frobenius $\varphi$ given by $\varphi(e)=d^{\alpha}e$.
\end{definition}
By construction, one has $\Os(-\alpha)\otimes \Os(-\beta)\simeq \Os(-\alpha-\beta)$. 
These $\Os(-\alpha)$ are all the possible rank $1$ modules with Frobenius, as the following lemma shows. 
\begin{lemma}\label{lem : rank1}
	If $r=1$, then $M\simeq \Os(-\alpha)$ for a unique $\alpha\in \mathbb Q_{\geq 0}$.
\end{lemma}
\proof
Let $x\in M$ be a generator and write $\varphi(x)=\lambda x$. Since $\Ocpflat$ is a valuation ring whose maximal ideal is generated by $d^{\alpha}$ for all $\alpha\in \Q_{> 0}$, there exist $\alpha\in \Q_{\geq 0}$ and $\mu\in \Ocpflat^*$ such that $\lambda=d^{\alpha}\mu$. Then $x:=e/\mu^{\frac{1}{p-1}}$ is a generator of $M$ such that $\varphi(x)=d^{\alpha}x$, so that $M\simeq \Os(-\alpha)$. For uniqueness, it is enough to observe that
$\Hom_{\varphi}(\Os(-\alpha),\Os(-\beta))\neq 0$ if and only if $\alpha\geq \beta$. 
\endproof

\begin{proposition}\label{prop : existence filtration}
\begin{enumerate}
\item[]
	\item 	There exists a Frobenius-stable decreasing filtration $F^{\bullet}$ of $M$ such that $$\Gr_i(F^{\bullet}):=F^i/F^{i+1}\simeq \mathcal O(-\alpha_{i})$$ for some $\alpha_i\in \Q_{\geq 0}$.
	\item If $F^{\bullet}$ and $G^{\bullet}$ are two  Frobenius-stable decreasing filtrations such that 
	$$\Gr_i(F^{\bullet})\simeq \mathcal O(-\alpha_{i})\quad \text{and}\quad \Gr_i(G^{\bullet})\simeq \mathcal O(-\beta_{i})$$ for some $\alpha_i,\beta_i\in \Q_{\geq 0}$, then 
	$$\sum_i \alpha_i=\sum_i \beta_i.$$
\end{enumerate}
\end{proposition}

\begin{proof}
\begin{enumerate}
	\item[] 
	\item We prove this by induction on $r$. The case $r=1$ is Lemma \ref{lem : rank1}, so assume $r>1$. By Lemma \ref{lem : notemptiness} below, the set  
\[
	\mathcal{E}(M,\varphi):=\{x\in M \mid \exists \text{ }i\in \Q_{\geq 0}\text{ such that } \varphi(x)=d^i x \text{ and }x\notin \mathfrak{m}^\flat M\}
	\]
	is finite and non-empty. In particular, there exists a minimal $i\in \Q_{\geq 0}$ such that there exists an $x\in M\setminus \mathfrak{m}^\flat M$ with $\varphi(x)=d^ix$. 
	Fix such an $x$; then the $\mathcal{O}_{\mathbb C_p^\flat}$-module $\langle x\rangle$ generated by $x$ is isomorphic to $\mathcal O(-i)$. Thus, by induction it is enough to show that $M/\langle x\rangle$ is torsion-free (hence free, since $\Ocpflat$ is a valuation ring).
	It is enough to show that if $y\in M$ and $j\in \Q_{\geq 0}$ are such that $d^{j}y=\lambda x$ for some $\lambda\in \Ocpflat$, then $a:=v_d(\lambda)\geq j$. Upon rescaling $j$, we can assume that $y\in M\setminus \mathfrak{m}^\flat M$, and upon multiplying it by a unit, we can assume that $\lambda=d^a$. Then we have
	$$d^{pj}\varphi(y)=\varphi(d^jy)=\varphi(\lambda x)=\lambda^pd^{i}x=\lambda^{p-1}d^{i+j}y=d^{a(p-1)+i+j}y,$$
	hence $\varphi(y)=d^{a(p-1)+i-j(p-1)}y$. By the minimality of $i$ and since $y\in M\setminus \mathfrak{m}^\flat M$, we get $a(p-1)+i-j(p-1)\geq i$, hence $a\geq j$.
	\item Observe that 
	$$\mathcal O\left(-\sum^r_{i=1}\alpha_i\right)\simeq \otimes^r_{i=1} \Gr^i(F^{\bullet})\simeq \wedge^rM\simeq \otimes^r_{i=1} \Gr^i(G^{\bullet})\simeq \mathcal O\left(-\sum^r_{i=1}\beta_i\right),$$
	so the conclusion follows from Lemma \ref{lem : rank1}. \qedhere
\end{enumerate}
\end{proof}
\begin{lemma}\label{lem : notemptiness}
	The set  
\[
	\mathcal{E}(M,\varphi):=\{x\in M \mid \exists \text{ }i\in \Q_{\geq 0}\text{ such that } \varphi(x)=d^i x \text{ and }x\notin \mathfrak{m}^\flat M\}
\]
	is finite and non-empty.
\end{lemma}
\proof
Recall that by \cite[Lemma 4.13, p.128]{Milne}, $M\left[\frac{1}{d}\right]^{\varphi=1}$ is an $\mathbb F_p$-vector space of dimension $r$, in particular it is finite and non-empty. 
Hence, to prove the lemma it is enough to construct a bijection
			$$f:\mathcal{E}(M,\varphi) \rightarrow M\left[\frac{1}{d}\right]^{\varphi=1}.$$
			For $x\in \mathcal{E}(M,\varphi)$, we define $f(x):=\frac{x}{d^{i/(p-1)}}$, where $i\in \mathbb Q_{\geq 0}$ is the unique rational such that $\varphi(x)=d^{i}x$. Since the map is well defined and surjective by definition, we just need to show injectivity. 		
        If $x,y\in \mathcal{E}(M,\varphi)$ map to the same element, then 
		\[
		\frac{x}{d^{i/(p-1)}}=\frac{y}{d^{j/(p-1)}} \quad \text{hence}\quad  d^{j/(p-1)} x= d^{i/(p-1)} y.
		\]
		If $i=j$, we immediately get $x=y$. If $i\ne j$, then without loss of generality we can assume $i>j$ and hence 
		\[
		x=d^{(i-j)/(p-1)} y.
		\]
		This implies $x\in \mathfrak{m}^\flat M$, which is a contradiction. \endproof

  \begin{definition}
	The total slope of $M$ is $$\TS (M):=\sum\alpha_i,$$ where the $\alpha_i\in \Q_{\geq 0}$ are the ones appearing in the graded quotients for any Frobenius-stable decreasing filtration $F^{\bullet}$ of $M$ such that $\Gr_i(F^{\bullet})\simeq \mathcal O(-\alpha_{i})$ for some $\alpha_i\in \Q_{\geq 0}$. The total slope is well defined thanks to Proposition \ref{prop : existence filtration}(2). 
\end{definition}

\begin{example}\label{ex : slope}
By contrast, the set of ``slopes'', i.e.\ the set of $\alpha_i$ appearing in the filtration, is not well defined, contrary to what one could expect by analogy with the crystalline situation (see e.g.\ \cite{Katzslope}), as the following example shows. To simplify the computation, we assume that $p=2$. 
Let $M=\Ocpflat\oplus \Ocpflat$ be endowed with a Frobenius whose matrix with respect to the standard basis $\mathcal B$ of $M$ is given by 
		$$M_{\mathcal B}(\varphi)=\begin{bmatrix}
    d^2    & 1\\
    0      & d^{1/2} \\
\end{bmatrix}.$$
With respect to the filtration induced by $\mathcal B$, the ``slopes'' are $1/2$ and $2$.
On the other hand, by Hensel's lemma, there exists $x\in \Ocpflat$ such that $dx^2+x+1=0$ and such that the reduction of $x$ modulo $d$ is $1$. In particular, $x$ is a unit. Hence, the elements $(x,d^{1/2})$ and $(0,1)$ form a new basis $\mathcal C$ of $M$. One computes that 
	$$M_{\mathcal C}(\varphi)=\begin{bmatrix}
    d    & x^{-1}\\
    0      & xd^{3/2} \\
\end{bmatrix},$$
so that with respect to the filtration induced by $\mathcal C$, the ``slopes'' are $1$ and $3/2$. Hence, the set of ``slopes'' is not well defined, as claimed.  As an additional remark, observe that in the first filtration the ``slopes'' are increasing, while in the second one they are decreasing (and neither of the filtrations splits). Thus, even the order of the ``slopes'' is not well defined. 
\end{example}

\begin{remark}\label{rem : matricesnewton}
	The existence of a filtration as in Proposition \ref{prop : existence filtration} is equivalent to the existence of a basis $\mathcal B$ of $M$ such that the matrix $M_{\mathcal B}(\varphi)$ is upper triangular with $d^{\alpha_i}$ on the diagonal. In particular, $\TS (M)$ is the $d$-adic valuation of the determinant of $M_{\mathcal C}(\varphi)$ for any basis $\mathcal C$ of $M$.
\end{remark} 

    \subsubsection{Hodge polygon}
Since $M/\varphi(M)$ is finitely presented and  $\Os_{\mathbb{C}_p^\flat}$ is a valuation ring whose maximal ideal is generated by $(d^{a})_{a\in \Q_{> 0}}$, by \text{\cite[\href{https://stacks.math.columbia.edu/tag/0ASP}{Tag 0ASP}]{stacks-project}},  we have 
	\begin{equation}\label{eq : definition of betai}
		\frac{M}{\varphi(M)}\simeq \bigoplus_i \frac{\Os_{\mathbb{C}_p^\flat}}{d^{\beta_i}}
			\end{equation}
	for unique $\beta_i\in \mathbb{Q}$ and $0\leq  \beta_1\leq \beta_2\leq \dots \leq \beta_r$. We call $\{\beta_i\}$ the Hodge slopes of $M$. 
	\begin{remark}\label{rem : matricesnhodge}
		By \cite[\href{https://stacks.math.columbia.edu/tag/0AST}{Tag 0AST}]{stacks-project}, there are two bases $\mathcal B$, $\mathcal B'$ of $M$ such that the matrix of $\varphi$ associated to $\mathcal B$, $\mathcal B'$ is diagonal with entries $d^{\beta_1},\dots,d^{\beta_r}$. In particular, one has 
$$\sum_{i=1}^r\beta_i=v_d(\Det(M_{\mathcal C}(\varphi))),$$
for any basis $\mathcal C$ of $M$. 
	\end{remark}
	\begin{definition}
	Let $P_0:=(0,0)\in \mathbb R^2$ and for $1\leq j\leq r$, let 
	$$P_j:=(j,\sum_{i=1}^j \beta_i)\in \mathbb R^2$$
		We define the \emph{Hodge polygon} of $M$ to be the union, for $0\leq j\leq r-1$, of the segments that join $P_j$ and $P_{j+1}$. The \emph{height} $h(M)$ of the Hodge polygon of $M$ is the height of the end point, that is, the sum of all $\beta_i$'s. 
			\end{definition}
		\begin{lemma}\label{lem : hodgevsnewton}
			The total slope of $M$ is equal to the height of the Hodge polygon of $M$. 
		\end{lemma}
		\proof
This follows from the fact that both are the $d$-adic valuation	of the determinant	of $M_{\mathcal B}(\varphi)$ for any basis $\mathcal B$ of $\varphi$, see Remarks \ref{rem : matricesnewton} and \ref{rem : matricesnhodge}. 
		\endproof
\subsection{Hodge polygon and Hodge numbers}\label{subsec : geometric Hodge polygon}
Let $\mathcal X$ be a smooth proper scheme over $\Ocp$ and $\widehat{\mathcal X}$ its formal $p$-adic completion. We write $Y$ for $\mathcal X_k$ and $X$ for $\mathcal X_{\mathbb C_p}$.
We let 
$$h^{a,b}:=\Dim_{\mathbb C_p}(\H^b(X,\Omega^a_{X/\mathbb C_p}))\quad \text{and} \quad h^n:=\Dim_{\mathbb C_p}(\H^n_{\dr}(X/\mathbb C_p))=\Dim_{\mathbb Q_p}(\H^n_{\et}(X,\Qp))$$
If \ref{eq : equalityhodgenumber} holds, then $\H^n(\Delta/p)$ is a finite free $\Ocpflat$-module of rank $h^n$ by Proposition \ref{prop : basicpropertiesprismatic}. Hence, $\H^n(\Delta/p)$ has a total Newton slope $\TS (\H^n(\Delta/p))$, a Hodge polygon and a Hodge height $h(\H^n(\Delta/p))$. We now explain how to compute them from the geometry of $\mathcal X$.
\begin{definition}
Let $Q_0:=(0,0)\in \mathbb R^2$ and, for $1\leq j\leq h^n$, let 
	$$Q_j:=\Biggl(\sum_{i=0}^j h^{i,n-i},\sum_{i=0}^j i \cdot h^{i,n-i}\Biggr)\in \mathbb R^2.$$
		We define the \emph{geometric Hodge polygon} of $\H^n_{\dr}(\mathcal X/\Ocp)$ to be the  union, for $0\leq j\leq r-1$, of the segments that join $Q_j$ and $Q_{j+1}$.
\end{definition}
\begin{proposition}\label{prop : hodgevsgeometrichodge}
Assume that \ref{eq : equalityhodgenumber} holds. 
Then the Hodge polygon of $\H^n(\Delta/p)$ coincides with the geometric Hodge polygon of $\H^n_{\dR}(\mathcal X/\Ocp)$. 
\end{proposition}
Combining Proposition \ref{prop : hodgevsgeometrichodge} and Lemma \ref{lem : hodgevsnewton}, we get Theorem \ref{thm : HodgevsNewton}. 
\proof[Proof of Proposition \ref{prop : hodgevsgeometrichodge}]
By definition, we need to show that that there exists an isomorphism
$$\H^n(\Delta/p)/\varphi(\H^n(\Delta/p))\simeq \bigoplus_{j=0}^{n}(\Ocpflat/d^{j})^{h^{j,n-j}},$$
which in turn is equivalent to showing that there exists an isomorphism
$$\H^n(\Delta/p)/\phi(\H^n(\Delta^{(1)}/p))\simeq \bigoplus_{j=0}^{n}(\Ocpflat/d^{j})^{h^{j,n-j}}.$$

This is equivalent to showing that, for every $i\in \Z_{\geq 0}$, there exists an isomorphism 
$$\H^n(\Delta/p)/(\phi(\H^n(\Delta^{(1)}/p)),d^i)\simeq \bigoplus^{i-1}_{j=0}(\Ocpflat/d^{j})^{h^{j,n-j}}\oplus (\Ocpflat/d^{i})^{\sum_{j\geq i} h^{j,n-j}}.$$

Choose bases of $\H^n(\Delta^{(1)}/p))$ and $\H^n(\Delta/p)$ so that the matrix of $\phi$ with respect to these bases is diagonal (see Remark \ref{rem : matricesnhodge}) with entries $\{d^{\beta_1}, \dots, d^{\beta_r}\}$, where $\{\beta_1,\dots, \beta_r\}$ are the Hodge slopes of $\H^n(\Delta/p)$. 
Then, the commutative diagram with exact rows and surjective vertical arrows
	\begin{center}
	\begin{tikzcd}
0\arrow{r}&	\H^n(\Delta^{(1)}/p))\arrow{r}{\phi}\arrow{d}&\H^n(\Delta/p)\arrow{r}\arrow{d}& \frac{\H^n(\Delta/p)}{\phi(\H^n(\Delta^{(1)}/p))}\arrow{r}\arrow{d} & 0 \\
0\arrow{r}& \frac{\H^n(\Delta^{(1)}/p))}{\phi^{-1}(d^i\H^n(\Delta/p))}\arrow{r}{\phi}&\frac{\H^n(\Delta/p)}{d^i\arrow{r}}& \frac{\H^n(\Delta/p)}{(\phi(\H^n(\Delta^{(1)}/p)), d^i)}\arrow{r}&0
	\end{tikzcd}
\end{center}
induces a commutative diagram with exact rows
	\begin{center}
	\begin{tikzcd}
0\arrow{r}& \frac{\H^n(\Delta^{(1)}/p)}{\phi^{-1}(d^i\H^n(\Delta/p))}\arrow["\phi"']{d}{\simeq}\arrow{r}{\phi}&\frac{\H^n(\Delta/p)}{d^i}\arrow{r}\arrow{d}{\simeq}& \frac{\H^n(\Delta/p)}{(\phi(\H^n(\Delta^{(1)}/p)),d^i)}\arrow{r}\arrow{d}{\simeq}&0\\
0\arrow{r}& \bigoplus_{\beta_j<i}\frac{d^{\beta_j}\Ocpflat}{d^{i}}\arrow{r}&\frac{\Ocpflat^{r}}{d^i} \arrow{r}& \bigoplus_{\beta_j<i}\frac{\Ocpflat}{d^{\beta_j}}\oplus \bigoplus_{\beta_j\geq i}\frac{\Ocpflat}{d^{i}}\arrow{r}&0, 
	\end{tikzcd}
	\end{center}
	in which the vertical arrows are isomorphisms, the bottom left arrow is the natural inclusion and $r:=h^n$.
	Hence, it is enough to show that there exists an isomorphism $$\H^n(\Delta^{(1)}/p)/\phi^{-1}(d^i\H^n(\Delta/p))\simeq \bigoplus^{i-1}_{j=0}(d^{j}\Ocpflat/d^{i})^{h^{j,n-j}}.$$
		We prove this by induction on $i$. The case $i=0$ is trivial, since all the groups involved are zero, so assume $i\geq 1$.
	By Proposition \ref{prop : nygaard}, one has 
	$$\phi^{-1}(d^i\H^n(\Delta/p))=\H^n(N^{\geq i}/p)\quad \text{and}$$
	$$\quad \H^n(\Delta^{(1)}/p)/\H^n(N^{\geq i}/p)\simeq  \H^n((\Delta^{(1)}/p)/(N^{\geq i}/p))\simeq \H^n(\Delta^{(1)}/p)/\phi^{-1}(d^i\H^n(\Delta/p)).$$
Thanks to Proposition~\ref{prop : nygaard}, we have a commutative diagram with exact rows and injective vertical maps (cf. \eqref{d})
\begin{center}
	\begin{tikzcd}[column sep=tiny]	
0\arrow{r}& \H^n((N^{\geq i-1}/p)/(N^{\geq i}/p)) \arrow{d}{\phi_{i-1}}\arrow{r}{}&\H^n(\Delta^{(1)}/p)/\H^n(N^{\geq i}/p)\arrow{r}\arrow{d}{\phi}& \H^n(\Delta^{(1)}/p)/\H^n(N^{\geq i-1}/p)\arrow{r}\arrow{d}{\phi}&0\\
0\arrow{r}& \H^n(\Delta/p)/d\arrow{r}{d^{i-1}}& \H^n(\Delta/p)/d^i\arrow{r} & \H^n(\Delta)/d^{i-1}\arrow{r} & 0
	\end{tikzcd}
	\end{center}
which can be identified, via the choice of the basis and the induction hypothesis, with
	\begin{center}
	\begin{tikzcd}[column sep=tiny]	
	& 0\arrow{d} & 0\arrow{d} & 0\arrow{d}\\
0\arrow{r}& \H^n\left(\frac{N^{\geq i-1}/p}{N^{\geq i}/p}\right) \arrow{d}\arrow{r}{}&\bigoplus_{\beta_j\geq i-1}\frac{d^{\beta_j}\Ocpflat}{d^{i}} \oplus  \bigoplus^{i-2}_{j=0}\left(\frac{d^{j}\Ocpflat}{d^{i}}\right)^{h^{j,n-j}}\arrow{r}\arrow{d}{}& \bigoplus^{i-2}_{j=0}\left(\frac{d^{j}\Ocpflat}{d^{i-1}}\right)^{h^{j,n-j}}\arrow{r}\arrow{d}{}&0\\
0\arrow{r}& (\Ocpflat/d)^r\arrow{r}{d^{i-1}}& (\Ocpflat/d^{i})^r\arrow{r} & (\Ocpflat/d^{i-1})^r\arrow{r} & 0,
	\end{tikzcd}
	\end{center}
	in which the upper right horizontal map is induced by the canonical projection
	$d^{j}\Ocpflat/d^{i}\rightarrow d^{j}\Ocpflat/d^{i-1}$. It remains to show that $\H^n((N^{\geq i-1}/p)/(N^{\geq i}/p))$ is a finite free $\Ocpflat/d$-module of rank $\sum^{i-1}_{j=0} h^{j,n-j}$. 
    To prove this, observe that, by Theorem \ref{thm : BS}, one has an isomorphism
	$$(N^{\geq i-1}/p)/(N^{\geq i}/p)\simeq \tau_{\leq i-1}(\overline \Delta/p).$$
	On the other hand, recall from Theorem \ref{thm : BS} that $\mathcal H^b(\overline \Delta/p)\simeq \Omega^b_{\mathcal X/p/(\Ocpflat/d)}$, so that $\H^a(\mathcal H^b(\overline \Delta/p))$ is a free $\Ocpflat/d$-module of rank $h^{b,a}$, by Lemma \ref{lem : Hodgegroups}. The conclusion now follows from the degeneration of the conjugate spectral sequence for $\overline \Delta/p$, 
see Proposition \ref{prop : basicpropertiesprismatic}. 
\endproof

\begin{remark}
    It follows from Proposition~\ref{prop : hodgevsgeometrichodge} that, in this setting, the Hodge slopes $\beta_i$ lie in $\{0,\dots,n\}$. 
\end{remark}

\section{Proof of the main result}\label{sec : prooffinal}
In this section we prove Theorems \ref{thm : jumpofdimension} and \ref{thm : main}, collecting the fruits of the work done in the previous sections.  We start in Section \ref{subsec : semilinear} by proving a useful semi-linear algebra lemma. The proofs of Theorems \ref{thm : jumpofdimension} and \ref{thm : main} are then contained in Sections \ref{subsec : proofofjump} and \ref{subsec : proofofmain}, respectively. 
\subsection{A semi-linear algebra lemma}\label{subsec : semilinear}
\begin{lemma}\label{lem : d-fixedpoints}
Let $M$ be a finite free $\mathcal O_{{\mathbb{C}_p^{\flat}}}$-module of rank $r\in \Z_{\geq 1}$ equipped with a Frobenius semi-linear map $\varphi:M\rightarrow M$ that becomes an isomorphism after inverting $d$.
	Suppose that $\TS (M)=ir$ for some $i\in\mathbb{Z}_{\geq 0}$. Then 
	$$
		\Dim_{\mathbb F_p} (M^{\varphi=d^i})=r
	 \quad \text{if and only if}\quad 
		 		M\simeq \bigoplus_{i=1}^r \Os(-i). $$
		\end{lemma}
\proof
The reverse implication being clear, we prove the forward implication.  To do this, we argue by induction on $r$, the case $r=1$ being immediate from the definition of total slope.  
So assume that $r>1$ and $\Dim_{\mathbb F_p} (M^{\varphi=d^i})=r$. By Proposition \ref{prop : existence filtration}, there exists a Frobenius-stable decreasing filtration $F^{\bullet}$ of $M$ and $a_1,\dots, a_r\in \mathbb Q_{\geq 0}$ such that
$$\Gr_j:=F^j/F^{j+1}\simeq \mathcal O(-a_j)\quad \text{and}\quad \sum_{j=1}^r a_j=ir.$$
Since 
$$r=\Dim_{\mathbb F_p} (M^{\varphi=d^i})\leq \sum_{j=1}^r\Dim_{\mathbb F_p} (\Gr_j^{\varphi=d^i})\quad \text{and}\quad \mathcal O(-a_j)^{\varphi=d^i}\neq 0 \Leftrightarrow a_j\leq i,$$
the assumption $\Dim_{\mathbb F_p} (M^{\varphi=d^i})=r$ implies that $a_j\leq i$ for every $j$. Since $\sum_{j=1}^r a_j=ir$, this in turn implies that $a_j=i$ for every $j$. 
Consider the commutative diagram with exact rows and columns
\begin{center}
	\begin{tikzcd}
0\arrow{r}&(F^1)^{\varphi=d^i}\arrow{r}\arrow[hook]{d}&M^{\varphi=d^i}\arrow{r}\arrow[hook]{d}& \mathcal O(-i)^{\varphi=d^i}=\mathbb F_p\arrow[hook]{d}\\
0\arrow{r}&F^1\arrow{r}\arrow{d}{d^i-\varphi}&M\arrow{r}\arrow{d}{d^i-\varphi}& \mathcal O(-i)\arrow{r}\arrow{d}{d^i-\varphi}& 0.\\
0\arrow{r}&F^1\arrow{r}&M\arrow{r}& \mathcal O(-i)\arrow{r}& 0.
	\end{tikzcd}
\end{center}
Since $\Dim_{\mathbb F_p}((F^1)^{\varphi=d^i})\leq \rank_{\Ocpflat}(F^1)=r-1$ and $\Dim_{\mathbb F_p} (M^{\varphi=d^i})=r$, the upper row shows that $\Dim_{\mathbb F_p}((F^1)^{\varphi=d^i})=r-1$. Since $\TS (F^1)=i(r-1)$, by induction $F^1\simeq \bigoplus_{i=1}^{r-1} \Os(-i)$, hence it is enough to show that the surjective map $M\rightarrow \mathcal O(-i)$ admits a $\varphi$-equivariant section. Since  $\Dim_{\mathbb F_p}((F^1)^{\varphi=d^i})=r-1$ and $\Dim_{\mathbb F_p}(M^{\varphi=d^i})=r$, the map $M^{\varphi=d^i}\rightarrow \mathcal O(-i)^{\varphi=d^i}$ is surjective. The $\Ocpflat$-module $\mathcal O(-i)$ is generated by an element $e$ such that $\varphi(e)=d^ie$, thus there exists an element $x\in M^{\varphi=d^i}\subseteq M$ mapping to $e$ and the map $\mathcal O(-i)\rightarrow M$ sending $e$ to $x$ gives a $\varphi$-equivariant splitting of the surjection $M\rightarrow \mathcal O(-i)$. 
\endproof
\begin{example}\label{ex : slopedfixedpoint}
		In Lemma \ref{lem : d-fixedpoints}, to conclude that $M\simeq \bigoplus_{i=1}^r \Os(-i)$ it is not enough to assume that $\Dim_{\mathbb F_p} (M^{\varphi=d^i})=r$. Some assumptions on $\TS (M)$ are necessary, as the following example shows. Let $M=\Ocpflat\oplus \Ocpflat$ be endowed with a Frobenius $\varphi$ whose matrix with respect to the standard basis $\mathcal B$ of $M$ is given by 
		$$M_{\mathcal B}(\varphi):=\begin{bmatrix}
    0    & d\\
    1       & 0 \\
\end{bmatrix}$$
Then $\TS (M)=1$, so that $M\not\simeq \Os(-1)\oplus \Os(-1)$, but 
$$M^{\varphi=d}=\{(a^pd^{\frac{p}{p^2-1}},ad^{\frac{1}{p^2-1}})\in \Ocpflat\oplus \Ocpflat \text{ for } a\in \mathbb F_{p^2}\}$$
is a 2-dimensional $\mathbb F_p$-vector space. 
		\end{example}

\subsection{Proof of Theorem \ref{thm : jumpofdimension}}\label{subsec : proofofjump}
Arguing by contradiction, we assume that 
$$\Dim_{\mathbb F_p}(\H^{2n}(\Delta/p)^{\varphi=d^n})=\rank_{\mathcal O_{\mathbb C^{\flat}_p}}(\H^{2n}(\Delta/p)).$$
By Theorem \ref{thm : HodgevsNewton}, one has  
$$\TS (\H^{2n}(\Delta/p))=\sum^{2n}_{j=0} j\cdot h^{j,2n-j}=n\cdot h^{2n}=n\cdot \rank_{\Ocpflat}(\H^{2n}(\Delta/p)),$$
so we can apply Lemma \ref{lem : d-fixedpoints} to deduce that 
$$\H^{2n}(\Delta/p)\simeq \mathcal O(-n)^{\rank_{\Ocpflat}(\H^{2n}(\Delta/p))}.$$
Hence, the Hodge polygon of $\H^{2n}(\Delta/p)$ is a straight line with slope $n$. By Proposition \ref{prop : hodgevsgeometrichodge}, this implies the same for the geometric Hodge polygon of $\H_{\dr}^{2n}(X/\mathbb C_p)$. By definition, this means that $\H^{2n}(X/\mathbb C_p)=\H^n(X,\Omega^n_{X/\mathbb C_p})$, that is, $\H^i(X,\Omega^{2n-i}_{X/\mathbb C_p})=0$ for $i\neq n$, which concludes the proof. 
\subsection{Proof of Theorem \ref{thm : main}}\label{subsec : proofofmain}
This follows from Theorems~\ref{thm : Brauerandhenselisation}, \ref{thm : prismaticinterpretationkernel} and \ref{thm : jumpofdimension}, and Proposition~\ref{prop : gb alt}.
\section{Abelian varieties and Kummer varieties}\label{sec : abelian varieties}
In this section we specialise to the case in which $X$ is an abelian variety. 
We start in Section \ref{subsec : prismaticdieu} by recalling the relationship between prismatic cohomology of abelian varieties, their p-torsion subgroups and Dieudonn\'e theory. We then apply this to products of elliptic curves (Section \ref{subsec: product}), to abelian varieties of positive $p$-rank and to their associated Kummer varieties (Section \ref{subsec : abelian varieties}). 

If $G$ is a finite locally free group scheme over $\mathcal O_{\mathbb C_p}$ we write 
$$0\rightarrow G^0\rightarrow G\rightarrow G^{\et}\rightarrow 0$$
for the connected-\'etale sequence of $G$, and we denote by $G^{\vee}$ the Cartier dual of $G$. 
\subsection{Prismatic Dieudonn\'e theory}\label{subsec : prismaticdieu}
Let $\mathcal C$ be the category of free $\Ocpflat$-modules $M$ endowed with a $\varphi$-linear Frobenius $\varphi_M:M\rightarrow M$ which becomes an isomorphism after inverting $d$. 
The following theorem summarises the main results we need from \cite{Dieudonneprismatic}.
\begin{theorem}\label{thm : prismaticDieudonne}
	There exists an exact contravariant fully faithful functor
	$$\mathbb D:\{\text{p-torsion finite locally free $\Ocp$- group schemes}\}\rightarrow \mathcal C$$
	 such that
	\begin{enumerate}
	\item $\mathbb D(\mu_p)=\mathcal O(-1)$;
	\item $\mathbb D(\mathcal G^{\vee})\simeq \mathbb D(\mathcal G)^{\vee}(-1)$, where $\mathbb D(\mathcal G)^{\vee}$ is the dual of $\mathbb D(\mathcal G)$;
		\item $\mathbb D(\Z/p)=\mathcal O(0)$;
\item if $\mathcal A/\Ocp$ is an abelian scheme of dimension $g$, then $\H^n(\mathcal A,\Delta/p)\simeq \wedge^n\mathbb D(\mathcal A[p])$ and it has rank $\binom{2g}{n}$;
\item $\mathbb D(G)^{\varphi=1}\otimes \Ocpflat=\mathbb D(G^{\et})$.
\end{enumerate}
\end{theorem}
\proof
The existence of the functor follows from \cite[Theorem 5.4]{Dieudonneprismatic}  and the fact that a $p$-torsion $A_{\inf}$-module has projective dimension $\leq 1$ if and only if it is free as an $\Ocpflat$-module. Then $(1)$ follows from \cite[Proposition 4.77 and the preceding discussion]{Dieudonneprismatic}, since the map $f:\mathbb Z_p[[q-1]]\rightarrow A_{\inf}$ sending $q$ to $[\epsilon^{1/p}]$ satisfies $f\left(\frac{q^{p}-1}{q-1}\right)=\xi$, see also \cite[Corollary 1.3]{Mondalprismatic} and \cite[Notation 2.6.3]{BL}. Point $(2)$ is the combination of point $(1)$ and \cite[Proposition 4.73]{Dieudonneprismatic}, while $(3)$ follows from $(1)$ and $(2)$. Finally, $(4)$ is \cite[Corollaries 4.63 and 4.64]{Dieudonneprismatic} and (5) follows from \cite[Remark 4.91]{Dieudonneprismatic}. 
\endproof
\begin{example}\label{ex : 3elliptic curves}
Assume that $X$ is the triple product of an elliptic curve $E$ with N\'eron model $\mathcal E$. We claim that, in this situation, one has an equality
	$$\Dim_{\mathbb F_p}(\H^{4}(\mathcal X_{\Ocp},\Delta/p)^{\varphi=d^{3}})=\rank_{\mathcal O_{\mathbb C^{\flat}_p}}(\H^{4}(\mathcal X_{\Ocp},\Delta/p)).$$
    Indeed, by the K\"{u}nneth formula (see e.g.\ \cite[Corollary 3.31]{Dieudonneprismatic})
$$\H^{4}(\mathcal X_{\Ocp},\Delta/p)\simeq \bigoplus_{a+b+c=4}\H^a(\mathcal E,\Delta/p)\otimes \H^b(\mathcal E,\Delta/p)\otimes \H^c(\mathcal E,\Delta/p).$$
Since 
$$
\H^n(\mathcal E,\Delta/p)=\begin{cases}
\mathcal O(0),& \text{if $n=0$}\\
			\mathcal O(-1) , & \text{if $n=2$}\\
            \{0\}, & \text{if $n>2$}\\
		 \end{cases}
$$
by Theorem \ref{thm : prismaticDieudonne}, the claim follows from the fact that $\H^1(\mathcal E,\Delta/p)^{\varphi=d}=(\Z/p)^2$ by Example \ref{ex : fixed elliptic curve}. 
     \end{example}
\subsection{Products of elliptic curves}\label{subsec: product}
\begin{proof}[Proof of Proposition \ref{prop : elliptic curves}]
Recall that the K\"{u}nneth formula for \'etale cohomology 
$$\H_{\et}^2(X_{\mathbb C_p},\Z/p)\simeq \H_{\et}^2(Z_{\mathbb C_p},\Z/p)\oplus \H_{\et}^2(W_{\mathbb C_p},\Z/p) \oplus \H_{\et}^1(Z_{\mathbb C_p},\Z/p)\otimes \H_{\et}^1(W_{\mathbb C_p},\Z/p)\simeq $$$$\Z/p\oplus \Z/p \oplus \Hom_{\mathbb C_p}(Z[p],W[p])$$
induces, since $\NS(X_{\mathbb C_p})\simeq \Z\oplus \Z \oplus \Hom_{\mathbb C_p}(Z,W)$,
a natural isomorphism
$$\Br(X_{\mathbb C_p})[p]\simeq \Hom_{\mathbb C_p}(Z[p],W[p])/\Hom_{\mathbb C_p}(Z,W),$$
see~\cite[Proposition~3.3]{SZtorsionEC}. 
Since $\H^i(\mathcal Z_{\Ocp},\Delta/p)$ and $\H^i(\mathcal W_{\Ocp},\Delta/p)$ are torsion-free by Proposition \ref{prop : basicpropertiesprismatic}, the K\"{u}nneth formula for prismatic cohomology (see e.g.\ \cite[Corollary 3.31]{Dieudonneprismatic}) gives an isomorphism 
\begin{align*}
\H^2(\mathcal X_{\Ocp},\Delta/p)&\simeq \H^2(\mathcal Z_{\Ocp},\Delta/p)\oplus \H^2(\mathcal W_{\Ocp},\Delta/p) \oplus \H^1(\mathcal Z_{\Ocp},\Delta/p)\otimes \H^1(\mathcal W_{\Ocp},\Delta/p)\\
&\simeq\mathcal O(-1)\oplus \mathcal O(-1)\oplus \H^1(\mathcal Z_{\Ocp},\Delta/p)\otimes \H^1(\mathcal W_{\Ocp},\Delta/p),
\end{align*}
by Theorem~\ref{thm : prismaticDieudonne}.
Hence, by Theorems \ref{thm : Brauerandhenselisation} and \ref{thm : prismaticinterpretationkernel}, it is enough to show that 
$$(\H^1(\mathcal Z_{\Ocp},\Delta/p)\otimes \H^1(\mathcal W_{\Ocp},\Delta/p))^{\varphi=d}\simeq \Hom_{\Ocp}(\mathcal Z[p],\mathcal W[p]).$$

By Theorem \ref{thm : prismaticDieudonne}, this is equivalent to showing that 
$$(\H^1(\mathcal Z_{\Ocp},\Delta/p)\otimes \H^1(\mathcal W_{\Ocp},\Delta/p))^{\varphi=d}\simeq \Hom_{\mathcal C}(\mathbb D(\mathcal W_{\Ocp}[p]),\mathbb D(\mathcal Z_{\Ocp}[p])).$$
For this, observe that there are natural isomorphisms compatible with the Frobenius
\begin{align*} 
\H^1(\mathcal Z_{\Ocp},\Delta/p)\otimes \H^1(\mathcal  W_{\Ocp},\Delta/p)&\simeq \Hom(\H^1(\mathcal  W_{\Ocp},\Delta/p)^{\vee},\H^1(\mathcal Z_{\Ocp},\Delta/p))\\
& \simeq \Hom(\mathbb D(\mathcal W_{\Ocp}[p])^{\vee},\mathbb D(\mathcal Z_{\Ocp}[p]))\\
&\simeq \Hom(\mathbb D(\mathcal  W_{\Ocp}[p]^{\vee})(1),\mathbb D(\mathcal  Z_{\Ocp}[p]))\\
&\simeq \Hom(\mathbb D(\mathcal W_{\Ocp}[p])(1),\mathbb D(\mathcal Z_{\Ocp}[p])),
\end{align*}
where the first isomorphism follows from the fact that $\H^1(\mathcal Z_{\Ocp},\Delta/p)$ and $\H^1(\mathcal W_{\Ocp},\Delta/p)$ are finite free, the second and the third from Theorem \ref{thm : prismaticDieudonne} and the last from the fact that $\mathcal W_{\Ocp}[p]$ is self dual. To conclude, observe that 
\begin{align*}
    (\Hom(\mathbb D(\mathcal W_{\Ocp}[p])(1),\mathbb D(\mathcal Z_{\Ocp}[p])))^{\varphi=d}&=(\Hom(\mathbb D(\mathcal W_{\Ocp}[p]),\mathbb D(\mathcal Z_{\Ocp}[p])))^{\varphi=1}\\
    &=\Hom_{\mathcal C}(\mathbb D(\mathcal W_{\Ocp}[p]),\mathbb D(\mathcal Z_{\Ocp}[p])). \qedhere
\end{align*}
\end{proof}
The next corollary concerning fields of definition of interesting Brauer classes follows immediately from Proposition~\ref{prop : elliptic curves}.
\begin{corollary}\label{cor : fields of def}
    Let $X = Z \times W$ for elliptic curves $Z, W$ and let $L/K$ be a field extension with $\Hom_L(Z[p],W[p])=\Hom_{\mathbb C_p}(Z[p],W[p])$. Then the natural map
    $$\Br (X_L)[p]/\fil_0\Br(X_L)[p]\to \Br(X_{\overline K})[p]/\Br(X_{\overline K})[p]^{\gb}$$
    is surjective.
\end{corollary}
\begin{remark}\label{rmk : fields of def}
In Corollary~\ref{cor : fields of def}, one can take $L=K(Z[p],W[p])$, for example. For elliptic curves with complex multiplication by the ring of integers of an imaginary quadratic field $F$, this is related to the ray class field of $F$ with modulus $p$. Furthermore, if the CM field $F$ satisfies $\mathcal{O}_F^\times=\{\pm 1\}$, then~\cite[Proposition~2.2]{NewtonCM} shows that it suffices to take $L$ to be the compositum of $K$ with the ring class field of conductor $p$, an extension of $F$ of degree $h_F\cdot\left(p-\left(\frac{\delta_F}{p}\right)\right)$, where $h_F$ is the class number, $\delta_F$ is the discriminant, and $\left(\frac{\delta_F}{p}\right)$ is the Legendre symbol for $p$ odd and the Kronecker symbol for $p=2$.
\end{remark}
\begin{corollary}\label{cor : elliptic}
 Let $X = Z \times W$ for elliptic curves $Z, W$ with good reduction and N\'eron models $\mathcal Z$,$\mathcal W$, and write
 $$\Ker_{\mathcal Z}:=\Ker\Bigl(Z[p](\Ocp)\rightarrow \mathcal Z[p](\overline k)\Bigr),\quad \Ker_{\mathcal W}:=\Ker\Bigl(W[p](\Ocp)\rightarrow \mathcal W[p](\overline k)\Bigr).$$
 \begin{enumerate}
     \item Assume that the special fibres of $Z, W$ are ordinary. Then $\Br(X_{\overline K})[p]/\Br(X_{\overline K})[p]^{\gb}$ is one-dimensional, generated by any homomorphism sending an element in $\Ker_{\mathcal Z}$ to an element not in $\Ker_{\mathcal W}$.
     \item Assume that the special fibre of $W$ is ordinary but that of $Z$ is supersingular. Then $\Br(X_{\overline K})[p]/\Br(X_{\overline K})[p]^{\gb}$ is two-dimensional,  generated by the homomorphisms $Z[p]\rightarrow W[p]$ whose image is not contained in $\Ker_{\mathcal W}$.
     \item Assume that $Z=W$ and that the special fibre is supersingular. Then any element in $\End_{\mathbb C_p}(Z[p])$ whose characteristic polynomial has two distinct roots in $\mathbb F_p$ gives a non-zero element in $\Br(X_{\overline K})[p]/\Br(X_{\overline K})[p]^{\gb}$. In particular, the dimension of $\Br(X_{\overline K})[p]/\Br(X_{\overline K})[p]^{\gb}$ is at least one.
 \end{enumerate}
\end{corollary}
\begin{proof}
    \begin{enumerate}
    \item[]
        \item Since the special fibres of $Z,W$ are ordinary,
	$$\mathcal Z_{\Ocp}[p]\simeq \mu_p\times \Z/p\simeq \mathcal W_{\Ocp}[p].$$
	Hence, $\Hom_{\Ocp}(\mathcal Z[p], \mathcal W[p])$ can be written as
	$$ \End(\Z/p)\oplus \End(\mu_p)\oplus \Hom_{\Ocp}(\Z/p,\mu_p)\oplus \Hom_{\Ocp}(\mu_p,\Z/p).$$
	Observe that $\End_{\Ocp}(\Z/p)\simeq \End_{\Ocp}(\mu_p)\simeq \Z/p$, generated by the identity. Furthermore, $\Hom_{\Ocp}(\Z/p,\mu_p)\simeq \Z/p$, generated by the morphism sending $1$ to a primitive $p^{th}$-root of unity. By contrast, $\Hom_{\Ocp}(\mu_p,\Z/p)=0$, since $\mu_p$ is connected and $\Z/p$ is totally disconnected.
	Since $\Ker\Bigl(\mathcal Z[p](\Ocp))\rightarrow \mathcal Z[p](\overline k)\Bigr)=\mu_p(\mathbb C_p)$ and similarly for $\mathcal W$, Proposition \ref{prop : elliptic curves} shows that $\Br(X_{\overline K})[p]/\Br(X_{\overline K})[p]^{\gb}$ is one-dimensional, generated by any homomorphism sending an element in $\Ker\Bigl(Z[p](\Ocp)\rightarrow \mathcal Z[p](\overline k)\Bigr)$ to an element not in $\Ker_\mathcal W$.
    \item Since the special fibre of $W$ is ordinary but that of $Z$ is not,
	$$\mathcal W_{\Ocp}[p]\simeq \mu_p\times \Z/p,\quad \text{while }\ \mathcal Z_{\Ocp}[p] \text{ is connected.}$$
	Hence, 
	$$\Hom_{\Ocp}(\mathcal Z[p], \mathcal W[p])\simeq  \Hom_{\Ocp}(\mathcal Z[p], \Z/p)\oplus \Hom_{\Ocp}(\mathcal Z[p], \mu_p).$$
	Since $\mathcal Z_{\Ocp}[p]$ is connected, $\Hom_{\Ocp}(\mathcal Z[p], \Z/p)=0$, while 
	$$\Hom_{\Ocp}(\mathcal Z[p], \mu_p)=\Hom_{\Ocp}(\Z/p,\mathcal Z[p]^{\vee})\simeq \mathcal Z[p]^{\vee}(\Ocp)\simeq (\Z/p\Z)^2.$$

	Since $\Ker\Bigl(W[p](\Ocp)\rightarrow \mathcal W[p](\overline k)\Bigr)=\mu_p(\Ocp)$, Proposition \ref{prop : elliptic curves} shows that the quotient $\Br(X_{\overline K})[p]/\Br(X_{\overline K})[p]^{\gb}$ is two-dimensional, generated by the homomorphisms $Z[p]\rightarrow W[p]$ whose image is not contained in $\Ker_\mathcal W$.
    \item Since $Z$ has supersingular reduction, $\mathcal Z[p]_{\overline k}$ is not the product of two subgroups, hence also $\mathcal Z_{\Ocp}[p]$ is not. In particular, $\End_{\Ocp}(\mathcal Z[p])$ contains no idempotents apart from $ 0$ and $1$. Hence, Proposition \ref{prop : elliptic curves} shows that all non-zero multiples of the idempotent elements different from $0$ and $1$ in $\End_{\mathbb C_p}(Z[p])$ give non-zero elements in the quotient $\Br(X_{\overline K})[p]/\Br(X_{\overline K})[p]^{\gb}$. Since the scalars in $\End_{\mathbb C_p}(Z[p])\simeq M_2(\mathbb F_p)$ correspond to elements in $\Br(X_{\overline K})[p]^{\gb}$, any element in $\End_{\mathbb C_p}(Z[p])$ whose characteristic polynomial has two distinct roots in $\mathbb F_p$ gives a non-zero element in $\Br(X_{\overline K})[p]/\Br(X_{\overline K})[p]^{\gb}$. \qedhere
    \end{enumerate}
\end{proof}

\begin{remark}
    One can also obtain a result similar to Corollary~\ref{cor : elliptic}(3) given an isogeny of degree coprime to $p$ between elliptic curves $Z$ and $W$ with supersingular reduction.
\end{remark}
\begin{remark}
    Corollary~\ref{cor : elliptic} sheds light on existing results in the literature. For example, Corollary~\ref{cor : elliptic}(1) explains the existence of the  arithmetically interesting $3$-torsion Brauer class in~\cite[Example~4.12]{margherita2} and shows that all $3$-torsion classes with non-constant evaluation maps at primes above $3$ are scalar multiples of this one. See also Section~\ref{subsubsec : compare literature}.
\end{remark}
    \subsubsection{The CM case}\label{subsubsec : CM}
We now specialise to the situation in which $X=Z\times Z$ for an elliptic curve $Z$ with complex multiplication (CM) by the ring of integers $\mathcal O_L$ of an imaginary quadratic number field $L$. 
Writing $\delta$ for the discriminant of $\Q(\sqrt{-d})$, a generator for the ring of integers of $\Q(\sqrt{-d})$ is given by $\gamma=(\delta+\sqrt{\delta})/2$. 
Fix a $\Z/p$-basis for $Z[p](\mathbb C_p)$ of the form $P,\gamma P$ for some $P\in Z[p](\mathbb C_p)$. With respect to this basis, multiplication by $\gamma$ on $Z[p](\mathbb C_p)$
is given by the following matrix with entries in $\Z/p$:
\[\begin{pmatrix}0& \ \ \ \ \frac{\delta(1-\delta)}{4}\\ 1 & \delta\end{pmatrix}.\]
(Note that $\frac{\delta(1-\delta)}{4}\in\Z$ so it has a well-defined image in $\Z/p$ for all primes $p$, including $p=2$.)
\begin{corollary}\label{cor : CM}
    Let $X=Z\times Z$ for an elliptic curve $Z$ with CM by $\mathcal{O}_L=\Z[\gamma]$.
    \begin{enumerate}
        \item Suppose that $p\neq 2$ and let $\sigma\in \End_{\mathbb C_p}(Z[p])$ be induced by complex conjugation. Then $\sigma$ corresponds to a non-zero element in $\Br(X_{\overline K})[p]/\Br(X_{\overline K})[p]^{\gb}$.
        \item Suppose that $Z$ has supersingular reduction. If $p$ is odd or $p\nmid \delta$ then $\Br(X_{\overline K})[p]^{\gb}$ is trivial and the dimension of $\Br(X_{\overline K})[p]/\Br(X_{\overline K})[p]^{\gb}$ is two.
    \end{enumerate}
\end{corollary}
\begin{proof}
    \begin{enumerate}
        \item[] 
        \item If $Z$ has supersingular reduction, then this follows from Corollary~\ref{cor : elliptic}, since $\sigma^2=1$ and $p\neq 2$ (so $\sigma\neq 1$). Now suppose that $Z$ has ordinary reduction. By Corollary~\ref{cor : elliptic}(1), since $$\Ker\Bigl(\mathcal Z[p](\Ocp)\rightarrow \mathcal Z[p](\overline k)\Bigr)=\mu_p(\mathbb C_p)$$ is one-dimensional and stable by all the endomorphisms of $\mathcal Z_{\Ocp}[p]$, it is enough to show that $\mu_p(\mathbb C_p)$ is not stable under the action of $\sigma$. The group $\mu_p(\mathbb C_p)$ is stable by the action of $\gamma$, hence it is an eigenspace for $\gamma$. Since $Z$ has ordinary reduction, $\delta$ is a square modulo $p$. Since the eigenvalues of $\gamma$ are $\delta\pm\sqrt\delta$ and $\delta$ is negative, the eigenvectors for $\gamma$ are swapped by the action of $\sigma$. Consequently, $\mu_p(\mathbb C_p)$ is not preserved by the action of $\sigma$ and therefore $\sigma$ does not lift to an endomorphism of $\mathcal Z_{\Ocp}[p]$.
        \item Suppose that $p$ is odd or $p\nmid \delta$. Let $e_1, e_2$ be idempotent endomorphisms of $Z_{\mathbb C_p}[p]$ defined by the following matrices with respect to our fixed basis:
\[e_1=\begin{pmatrix}0 & 0\\
0 & 1
\end{pmatrix}, \ e_2= \begin{pmatrix} 0 & 1\\
0 & 1
\end{pmatrix}.\]
Since $e_1,e_2,\gamma, \Id$ is a basis of $\End_{\mathbb C_p}(Z[p])$ and $\gamma, \Id$ is a basis of $\End_{\mathbb C_p}(Z)$, it suffices to prove that all non-zero linear combinations $ae_1+be_2\in \End_{\mathbb C_p}(Z[p])$, with $a,b\in\mathbb F_p$, do not lift to $\End_{\Ocp}(\mathcal Z[p])$. 
The characteristic polynomial of $ae_1+be_2$ is $X(X-a-b)$ so if $a+b\neq 0$, the result follows from Corollary \ref{cor : elliptic}(3). Now suppose that $a+b=0$, so $\theta:=ae_1+be_2=\begin{pmatrix}0& b\\ 0 & 0\end{pmatrix}$. Note that $\theta$ lifts to $\End_{\Ocp}(\mathcal Z[p])$ if and only if $\theta+t\gamma$ lifts to $\End_{\Ocp}(\mathcal Z[p])$ for all $t\in\F_p$. Now, the characteristic polynomial of $\theta+t\gamma$ is 
 \[f(X)=X(X-t\delta)-t\left(b+t\left(\frac{\delta(1-\delta)}{4}\right)\right).\] 
 If we can find $t\in\F_p$ such that $f(X)$ splits into two linear factors over $\F_p$, then we are done by Corollary \ref{cor : elliptic}(3). A theorem of Shimura and Taniyama states that $E$ has supersingular reduction if and only if $p$ does not split in $\Q(\sqrt{-d})/\Q$. If $p$ ramifies in $\Q(\sqrt{-d})/\Q$ then $p\mid\delta$ (whereby $p$ is odd by assumption) and $f(X)$ becomes $X^2-tb$. If we take $t=b$ then $f(X)$ splits, as required. Now suppose that $p$ is inert in $\Q(\sqrt{-d})/\Q$. 
Then $p$ does not divide $\frac{1-\delta}{4}$ and we can take $t\in\F_p^\times$ such that $b+t\left(\frac{\delta(1-\delta)}{4}\right)=0$. Thus, $f(X)$ becomes $X(X-t\delta)$, which is split. 

The assertion on the dimension of $\Br(X_{\overline K})[p]/\Br(X_{\overline K})[p]^{\gb}$ is immediate, as $\Br(X_{\overline K})[p]\cong (\mathbb{Z}/p)^2$ by work of Grothendieck, 
see~\cite[Proposition~5.2.9]{CoSko}. \qedhere
    \end{enumerate}
\end{proof}

\begin{remark}
    If $X=Z_1\times Z_2$ where 
        $Z_1, Z_2$ are elliptic curves with CM by orders $\mathcal{O}_1,\mathcal{O}_2$ in $L$, then Corollary~\ref{cor : CM} applies with the additional constraint that the prime $p$ be coprime to $[\mathcal{O}_L:\mathcal{O}_i]$ for $i=1,2$.
\end{remark}
\subsubsection{Comparison with previous work.}\label{subsubsec : compare literature}
Products of CM elliptic curves (and their associated Kummer surfaces) have been the source of many examples of transcendental Brauer classes obstructing weak approximation in the literature to date. 
The results in~\cite{SkoroIeronymou, ErrataIeronymouSkoro}  
are obtained by relating diagonal quartic surfaces to products of curves with CM by $\mathbb{Z}[i]$. Products of elliptic curves over $\mathbb{Q}$ with CM by other maximal orders have been studied in~\cite{NewtonCM, NewtonMK3}. All these examples of transcendental Brauer classes obstructing weak approximation are related to complex conjugation as in Corollary~\ref{cor : CM}(1), possibly after composing with a geometric isomorphism $Z_{\mathbb{\bar{Q}}}\xrightarrow[]{\simeq} W_{\mathbb{\bar{Q}}}$. This is spelt out in~\cite[Proposition~4.6]{NewtonMK3} and \cite[Lemma~2.3.3]{Mohamed}, for example. 

Note that in the examples in~\cite{SkoroIeronymou, ErrataIeronymouSkoro, NewtonCM, NewtonMK3}, 
the relevant Brauer classes have order $p$ where $p$ is an odd prime of bad, but potentially good, reduction for the abelian surface. 
In the good reduction setting, Corollary~\ref{cor : CM}(1) shows that $p$-torsion Brauer classes coming from complex conjugation have non-constant (in fact, surjective -- see Proposition~\ref{prop : braueruseful}) evaluation maps at primes above $p$ over all finite extensions. In the ordinary reduction cases, (including~$p=5$ in \cite[Theorem~1.1]{SkoroIeronymou}, all cases of~\cite[Theorem~1.3]{NewtonMK3}, and the case $\ell=7$ of~\cite[Theorem~1.4]{NewtonMK3}), 
Corollary~\ref{cor : elliptic}(1) explains why the only arithmetically interesting Brauer classes come from (scalar multiples of) complex conjugation. Moreover, Corollary~\ref{cor : CM} shows that the constant evaluation map of the element $\mathcal{A}\in\Br(D)[3]$ for $D$ of type I in~\cite[Theorem~2.3]{ErrataIeronymouSkoro} is merely a temporary phenomenon -- the evaluation map will become surjective after passing to a finite extension where the surface attains good reduction above $3$. This example shows that good reduction is a necessary hypothesis in Proposition~\ref{prop : braueruseful}. 
\subsection{Abelian varieties of positive $p$-rank and associated Kummer varieties}\label{subsec : abelian varieties}
In this section we prove Theorem \ref{thm : abelian and kumm varieties}. Let $K/\Q_p$ be a $p$-adic field and $X/K$ an abelian variety of dimension $g\geq 2$ with  good reduction. Following \cite{SkoroZarhinKummer}, given a $K$-torsor $T$ for the $K$-group scheme $X[2]$, we define the associated $2$-covering of $X$ as the quotient $Y:=(X\times_K T)/X[2]$ by the diagonal action of $X[2]$. The antipodal involution $\iota_X$ on $X$ induces an involution $\iota_Y\colon Y\rightarrow Y$. Let $\sigma \colon Y'\rightarrow Y$ be the blowing-up of the $2^{2g}$-point closed subscheme $T\subseteq Y$. The involution $\iota_Y$ preserves $T$ and so gives rise to an involution $\iota_{Y'}$ on $Y'$. The Kummer variety attached to $Y$ is defined as the quotient $Y'/\iota_{Y'}=:\mathrm{Kum}(X_T)$ and it is a smooth proper variety. If $p$ is odd then $\mathrm{Kum}(X_T)$ has good reduction, since $X$ has good reduction.

We write $\pi\colon Y'\rightarrow \mathrm{Kum}(X_T)$ for the double covering whose branch locus is $E:=\sigma^{-1}(T)$. 
In \cite[Proposition~2.7]{SkoroZarhinKummer}, Skorobogatov and Zarhin prove that the morphisms $\pi$ and $\sigma$ induce an isomorphism of $\Gal(\bar{K}/K)$-modules 
\begin{equation}\label{eq : BrKum}
     \theta\colon \Br(\mathrm{Kum}(X_T)_{\bar{K}})\rightarrow \Br(X_{\bar{K}}).
\end{equation}
Note that $\theta$ induces an injection $\Br(\mathrm{Kum}(X_T)_{\bar{K}})[p]^{\gb}\hookrightarrow\Br(X_{\overline K})[p]^{\gb}$, by Theorem~\ref{thm : Brauerandhenselisation}.

\begin{proof}[Proof of Theorem \ref{thm : abelian and kumm varieties}]
    We begin with the statement for the abelian variety $X$. By Theorems \ref{thm : Brauerandhenselisation} and \ref{thm : prismaticinterpretationkernel}, it is enough to show that $$\rank_{\Ocpflat}(\H^2(\mathcal{X}_{\Ocp},\Delta/p))-\Dim_{\mathbb F_p}(\H^2(\mathcal{X}_{\Ocp},\Delta/p)^{\varphi=d})\geq 2g-1-e.$$
Since $e>0$, there is a surjection $\mathcal X[p]\rightarrow \mathbb Z/p$ with finite locally free kernel. Since $\mathcal X$ has a polarisation of degree prime to p, $\mathcal X[p]$ is self-dual and therefore there is a short exact sequence of finite locally free group schemes 
 $$0\rightarrow \mu_p\rightarrow \mathcal X[p]\rightarrow J\rightarrow 0,$$ 
 such that $\rank(J^{\et})=e$. 
 By Theorem \ref{thm : prismaticDieudonne}, this induces a short exact sequence
$$0\rightarrow \mathbb D(J)\rightarrow \mathbb D(\mathcal X[p])\rightarrow \mathcal O(-1)\rightarrow 0$$
in $\mathcal C$. Applying $\wedge^{2}$ and using Theorem \ref{thm : prismaticDieudonne}, we get an exact sequence 
$$0\rightarrow \wedge^2\mathbb D(J)\rightarrow \H^2(\Delta/p)\rightarrow \mathbb D(J)\otimes \mathcal O(-1)\rightarrow 0.$$
Hence, it is enough to show that 
$$\Dim_{\mathbb F_p}((\mathbb D(J)\otimes \mathcal O(-1))^{\varphi=d})=e .$$
But 
$$\Dim_{\mathbb F_p}((\mathbb D(J)\otimes \mathcal O(-1))^{\varphi=d})=\Dim_{\mathbb F_p}(\mathbb D(J)^{\varphi=1})\stackrel{(*)}{=}\rank(J^{\et})=e,$$
where (*) follows from Theorem \ref{thm : prismaticDieudonne}(5). 

We now consider the Kummer variety associated to the $X[2]$-torsor $T$.
By Theorem~\ref{thm : Brauerandhenselisation}, it is enough to prove that 
    \[
    \mathrm{Im}(\H^2(\mathrm{Kum}(X_T)_{\bar{K}},\Z/p)\rightarrow \H^2(\bar{K}(\mathrm{Kum}(X_T))^{\mathrm{sh}},\Z/p))\simeq \mathrm{Im}(\H^2(X_{\bar{K}},\Z/p)\rightarrow \H^2(\bar{K}(X)^{\mathrm{sh}},\Z/p)).
    \]
    As observed at the beginning of Section~\ref{subsec : proofofbrauerhens}, these maps factor through the $p$-torsion of the relevant Brauer groups, and these groups are isomorphic by~\eqref{eq : BrKum}. The result follows from the fact that $\bar{K}(X)/\bar{K}(\mathrm{Kum}(X_T))$ is a degree $2$ extension, and therefore the map $\H^2(\bar{K}(\mathrm{Kum}(X_T))^{\mathrm{sh}},\Z/p))\rightarrow \H^2(\bar{K}(X)^{\mathrm{sh}},\Z/p))$ is injective for $p>2$.
\end{proof}

\begin{remark}

     The problem of constructing even a single transcendental Brauer class giving an obstruction has hitherto been considered something of a challenge. Our results in the non-ordinary reduction setting give various instances of K3 surfaces having at least two independent transcendental elements of odd order $p$ that play a role in the Brauer--Manin obstruction to weak approximation. For example, one can take $X$ to be an abelian surface with $p$-rank one in Theorem~\ref{thm : abelian and kumm varieties}, or alternatively consider $\mathrm{Kum}(Z\times Z)$ where $Z$ is a CM elliptic curve with supersingular reduction and apply Corollary~\ref{cor : CM}(2) together with the last part of the proof of Theorem~\ref{thm : abelian and kumm varieties}.
\end{remark}
\printbibliography
 \end{document}